\documentclass{amsart}

\usepackage[utf8]{inputenc} 
\usepackage[T1]{fontenc}
\usepackage[osf]{Baskervaldx} 
\usepackage[baskervaldx]{newtxmath} 
\usepackage[cal= cm, bb = ams, frak = boondox, scr = rsfso]{mathalfa} 
\usepackage{anyfontsize} 

\usepackage{xcolor}

\usepackage{tikz-cd, tikz}
\usetikzlibrary{positioning} 

\usepackage{mathrsfs}

\usepackage{hyperref}

\usepackage{mathtools,amssymb,amsthm,amsfonts}

\usepackage{enumerate}

\reversemarginpar

\newtheorem{thm}{Theorem}
\newtheorem*{introHiggsthm}{Theorem \ref{thm:higgs description}}
\newtheorem*{introMainthm}{Theorem \ref{thm:main theorem}}
\newtheorem{lem}{Lemma}
\newtheorem{prop}{Proposition}
\newtheorem*{introLaxprop}{Proposition \ref{prop:geometric lax matrix}}
\newtheorem{cor}{Corollary}
\theoremstyle{definition}
\newtheorem{defn}{Definition}
\newtheorem{notn}{Notation}
\newtheorem{construction}{Construction}
\theoremstyle{remark}
\newtheorem{rmk}{Remark}

\DeclareMathOperator{\Tr}{Tr}
\DeclareMathOperator{\Jac}{Jac}
\DeclareMathOperator{\supp}{supp}
\DeclareMathOperator{\coker}{coker}
\DeclareMathOperator{\Tot}{Tot}
\DeclareMathOperator{\Cof}{Cof}
\DeclareMathOperator{\End}{End}
\DeclareMathOperator{\Aut}{Aut}
\DeclareMathOperator{\Hom}{Hom}
\DeclareMathOperator{\Ext}{Ext}
\DeclareMathOperator{\Res}{Res}
\DeclareMathOperator{\Pic}{Pic}
\DeclareMathOperator{\Spec}{Spec}
\DeclareMathOperator{\Hitchin}{Hitchin}
\DeclareMathOperator{\At}{At}
\DeclareMathOperator{\Cl}{Cl}
\DeclareMathOperator{\diag}{diag}
\DeclareMathOperator{\Bun}{Bun}
\newcommand{\PP}{\mathbb P}
\newcommand{\lra}{\longrightarrow}
\newcommand{\OO}{\mathcal O}
\newcommand{\thetachar}[2]{\theta \begin{bmatrix}#1 \\ #2\end{bmatrix}}

\title{Spectral Description of the Spin Ruijsenaars-Schneider System}
\author{Matej Penciak}
\address{Department of Mathematics\\Northeastern University\\Boston, MA 02215 USA}
\email{m.penciak@northeastern.edu}

\begin{document}

\begin{abstract}
Fix a Weierstrass cubic curve \(E\), and an element $\sigma$ in the Jacobian variety $\Jac E$ corresponding to the line bundle \(\mathcal L_\sigma \). We introduce a space \(\mathsf{RS}_{\sigma, n}(E, V)\) of pure 1-dimensional sheaves living in \(S_\sigma = \PP(\OO \oplus \mathcal L_\sigma)\) together with framing data at the \(0\) and \(\infty\) sections \(E_0, E_\infty \subset S_\sigma \). For a particular choice of \(V\), we show that the space \(\mathsf{RS}_{\sigma, n}(E, V)\) is isomorphic to a completed phase space for the spin Ruijsenaars-Schneider system, with the Hamiltonian vector fields given by tweaking flows on sheaves at their restrictions to \(E_0\) and \(E_\infty\). We compare this description of the RS system to the description of the Calogero-Moser system in \cite{MR2377220}, and show that the two systems can be assembled into a universal system by introducing a \(\sigma \rightarrow 0\) limit to the CM phase space. We also shed some light on the effect of Ruijsenaars' duality between trigonometric CM and rational RS spectral curves coming from the two descriptions in terms of supports of spectral sheaves.
\end{abstract}

\maketitle

\section{Introduction}\label{sec:Intro}
This paper deals with a completely integrable particle system called the Ruijsenaars-Schneider (RS) system. Until the works of Calogero, Moser, and others in the 1970s the only known integrable particle systems were essentially those of two orbiting bodies or an ensemble of simple harmonic oscillators. In \cite{1971JMP....12..419C} Calogero introduced a novel one dimensional, n-particle system with pairwise interacting potential and solved the quantum system. It was then Moser in \cite{Moser} who proved the corresponding classical system defined by the Hamiltonian
\[
H_{CM} = \frac{1}{2} \sum_{i=1}^{n} \frac{p_i^2}{m} + \frac{g^2}{2} \sum_{i \neq j} \frac{1}{(q_i - q_j)^2}
\]
defines a completely integrable Hamiltonian system. The way Moser was able to prove the integrability of the system was by using new techniques of Lax operators introduced in \cite{doi:10.1002/cpa.3160210503}. Since the proof of integrability a plethora of integrable particle systems have been constructed. The focus of this paper is on one of those systems introduced by Ruijsenaars and Schneider \cite{MR851627}.

The RS system can be viewed as a relativistic generalization of the Calogero-Moser (CM) system which we now recall. The \(n\)-particle CM system is defined as a system of interacting particles with complex positions \(q_1, \ldots q_n\) and momenta \(p_1, \ldots, p_n\) by the Hamiltonian 
\[
H_{CM} = \frac{1}{2} \sum_{i=1}^{n} \frac{p_i^2}{m} + \frac{g^2}{2} \sum_{i \neq j} f(q_i - q_j)
\]
where \(f(z)\) is one of three functions corresponding to the \emph{rational}, \emph{trigonometric}, or \emph{elliptic} cases of the CM system:
\[
f(z) = \begin{cases}
	\,\, \frac{1}{z^2} \,\,\,\, \text{rational} \\
	\\
	\,\, \frac{1}{\sin^2 z} \,\,\,\, \text{trigonometric}\\
	\\
	\,\, \wp(z) \,\,\,\, \text{elliptic.}
\end{cases}
\]
Where \(\wp(z)\) is the Weierstrass elliptic function defined with respect to a pair of periods \(\omega_1, \omega_2\) defining an elliptic curve \(E\). Because of the bi-periodicity of \(\wp(z)\), we can view the elliptic CM system as a system of particles living on the ellitpic curve \(E\). As described in Section \ref{subsec: line bundles} the rational and trigonometric potentials can be obtained as degenerations of \(\wp(z)\), or geometrically we can view the trigonometric and rational systems as degenerations of the underlying elliptic curve \(E\) into the nodal and cuspidal cubic curve.

Finally we may additionally introduce internal spin degrees of freedom \(u_i \in \mathbb C^k\), \(v_i \in (\mathbb C^k)^*\) by adding a factor of \(v_i(u_j)\) to the \(f(q_i - q_j)\) summand in \(H_{CM}\). These additional ingredients define the \emph{spin} CM system which is also known to be integrable. 

The RS system can be viewed as a relativistic analogue of the CM system. In the \(n\)-particle RS system, the positions \(q_i\) of the particles still live on an elliptic curve \(E\), or its nodal and cuspidal limits. On the other hand the momenta are replaced by the relativistic quantities known as \emph{rapidities} which take values in \(\mathbb C^*\), and hence can be written as \(e^{\theta_i}\) for \(\theta_i \in \mathbb C\). The formula for the RS Hamiltonian is 

\[
H_{RS} = mc^2 \sum_{i = 1}^n \cosh\left(\frac{\theta_i}{mc}\right) \prod_{k, \, k \neq i}^n(\nu + \kappa f(q_k - q_i))
\]
where \(f(z)\) is one of the three functions also used in the definition of the CM system above, and \(\nu, \kappa\) are physical constants that depend on \(c\). 

The sense in which the RS system is a relativistic version of the CM system can be made precise by evaluating of the non-relativistic limit of the RS Hamiltonian and recovering the CM Hamiltonian
\[
H_{CM} = \lim_{c \rightarrow \infty} H_{RS} - n m c^2.
\]
We can similarly introduce internal spin degrees of freedom for the RS particles to define the spin RS system in the same way as we did for CM.

The integrability of the CM and RS system follows by bringing the systems into \emph{Lax form}. The Lax form of a system is given by a pair of matrices \(L = L(\mathbf q, \mathbf p), M = M(\mathbf q, \mathbf p)\) for which the equations
\[
\frac{d}{dt} L = [M, L]
\]
are equivalent to the original equations of motion. The benefit of bringing a system to Lax form is that it yields constants of motion as the eigenvalues of the matrix L. Equivalently, we can choose a basis for the constants of motion given as traces of powers of L or the coefficients of the characteristic polynomial. With the additional structure of a classical \(r\)-matrix, these constants of motion can be proven to Poisson commute, and hence define a completely integrable system

It was in this way that Moser proved the integrability of the CM system, with the rational CM Lax matrix being given by the formula
\begin{equation}\label{eqn:CMLax}
	L_{CM}(\lambda) = \delta_{i,j} p_i + g(1- \delta_{i,j})(1/(q_i - q_j) - 1/\lambda)
\end{equation}
where here we have a Lax matrix with \emph{spectral parameter} \(\lambda\). The constants of motion when a spectral parameter is present are given as the coefficients of \(\lambda\) in the expansion of \(\Tr L(\lambda)^i\). In particular, the constant term of \(\Tr L_{CM}(\lambda)^2\) yields \(H_{CM}\) from above. 

When introduced, the RS system was proven to be integrable without a Lax matrix construction, but later it was provided a Lax matrix in \cite{ruijsenaars1987}. Since then many Lax matrices have been constructed, and we recall the various definitions in Section \ref{subsec: recollections on RS}. After having established integrability via the Lax formalism, a natural goal is to then provide a geometric picture of why the system is integrable. 

\subsection{Geometric perspectives on integrability} \label{subsec: geometric integrability}

Shortly after their introduction, the rational and trigonometric CM systems were given a geometric description by Kazhdan, Kostant, and Sternberg in \cite{doi:10.1002/cpa.3160310405}. The geometric picture given is that of Hamiltonian reduction on the space \(T^* \mathfrak{gl}_n\). The adjoint action of \(GL_n\) on its Lie algebra lifts to a Hamiltonian \(GL_n\) action on the cotangent bundle \(T^* \mathfrak{gl}_n \cong \mathfrak{gl}_n \times \mathfrak{gl}_n\) with moment map
\begin{align*}
	\mu: T^* \mathfrak{gl}_n &\lra \mathfrak{gl}_n\\
	\mu: (X, Y) &\mapsto [X, Y]
\end{align*}
where we identify \(\mathfrak{gl}_n\) with its dual via the non-degenerate Killing form. Let \(\mathscr O\) be the co-adjoint orbit of \(\mathfrak{gl}_n\) of minimal dimension which contains the matrix \(O\):
\[O = 
\begin{pmatrix}
	0 & 1 & \dots & 1 \\
	1 & 0 & \dots & 1 \\
	\vdots & \vdots & \ddots & \vdots \\
	1 & 1 & \dots & 0
\end{pmatrix}.
\]
Applying Hamiltonian reduction on \(T^* \mathfrak{gl}_n\) with respect to the coadjoint orbit \(\mathscr O\) yields a completion of the rational CM phase space. One way to see this is that we can use the action of \(GL_n\) to diagonalize the matrix \(X = \diag(q_1, \dots q_n)\) in the moment map, and solving the resulting equation 
\[
[X, Y] = O
\]
forces \(Y\) to be the rational CM Lax matrix (without spectral parameter). The choice of either diagonalizing \(X\) or \(Y\) makes no difference in the solution to the moment map equation, and hence this exhibits a self-duality of the CM system where the diagonalizations of \(X\) and \(Y\) can be simultaneously viewed as an expression for the constants of motion and the positions of CM particles. We return to this duality in Section \ref{subsec: TCM-RRS}.

If we instead consider the adjoint action of \(GL_n\) on itself, and lift it to a Hamiltonian action on \(T^* GL_n \cong GL_n \times \mathfrak{gl}_n\) then the moment map for the action becomes
\begin{align*}
	\mu: T^* GL_n &\lra \mathfrak{gl}_n\\
	\mu: (X, Y) &\mapsto X Y X^{-1} - Y.
\end{align*}
Now there is a structural difference in the choice of matrix to diagonalize. Diagonalizing \(Y = \diag(q_1, \ldots q_n)\) and solving the moment map equation for \(X\) we find the trigonometric CM Lax matrix. On the other hand if we diagonalize \(X = \diag( \exp(\theta_1), \dots, \exp(\theta_n))\) and instead solve for \(Y\) we obtain the Lax matrix for a new integrable system which we now recognize as the rational Ruijsenaars-Schneider  system. \footnote{This is not strictly the case. For \(\mathfrak{gl}_n\) the solution is given by the Lax matrix for the MacDonald hierarchy. The rational RS Lax matrix is recovered by considering \(\mathfrak{sl}_n\) instead of \(\mathfrak{gl}_n\). \cite{nekrasov1996}}

The trigonometric RS system can be obtained in a similar way, where we now consider the diagonal action of \(GL_n\) on \(GL_n \times GL_n\) and consider the group-valued moment map
\begin{align*}
	\mu: GL_n \times GL_n &\lra GL_n\\
	\mu: (X, Y) &\mapsto X Y X^{-1} Y^{-1}.
\end{align*}
Performing quasi-Hamiltonian reduction with respect to the adjoint orbit containing the matrix
\[
O' = I + \mathbf{u} \mathbf{v}^T
\]
for \(\mathbf{u}, \mathbf{v} \in \mathbb C^n\) recovers the trigonometric RS system \cite{Oblomkov_doubleaffine}. Note that \(O' \in GL_n(\mathbb C)\) for generic choice of \(\mathbf u, \mathbf v\) by Sylvester's determinant lemma which states that \(\det O' = 1 + \mathbf u^T \mathbf v\). See, for example, \cite{10.1093/imrp/rpn008} for an introduction to quasi-Hamiltonian reduction used in the context of the trigonometric RS system.
\begin{rmk}
	The spin integrable systems can be obtained by modifying the above constructions with the spaces \(T^*(\mathfrak{gl}_n \times \mathbb C^k), T^*(GL_n \times \mathbb C^k)\), and \((GL_n \times GL_n) \times (\mathbb C^k \times (\mathbb C^k)^*)\) inheriting a (quasi-) Hamiltonian \(GL_n\) action which acts trivially on the \(\mathbb C^k\) factor. 

	Finally, the Hamiltonian reduction in the rational CM case can also be interpreted as a symplectic reduction exhibiting the CM phase space as a quiver variety associated to the tadpole quiver
	
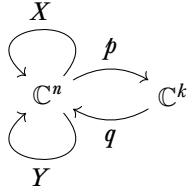
\begin{center}
		\begin{tikzpicture}
			\node (1) {\(\mathbb C^n\)};
			\node[right= 1 cm of 1] (2) {\(\mathbb C^k\)};
			\path[->]
			(1) edge[out = 50, in = 140, looseness = 6] node[above]{\(X\)} (1)
			(1) edge[out = 310, in = 220, looseness = 6] node[below]{\(Y\)} (1)
			(1) edge[bend left] node[above]{\(p\)} (2)
			(2) edge[bend left] node[below]{\(q\)} (1);
		\end{tikzpicture}
	\end{center}
	Similarly, the trigonometric RS system is given as the multiplicative quiver variety associated to the same tadpole quiver. 
\end{rmk}

In order to understand the full elliptic hierarchies, linear algebra alone does not suffice. The geometric description of the elliptic CM system came in \cite{Krichever1980} wherein the author describes the elliptic CM integrable system as a Hitchin-type system. The Lax matrix appears as the Higgs field in the Hitchin description, and integrability comes from an interpretation of the Hitchin integrable system on \(T^* \Bun_n(E)\) for an elliptic curve \(E\). 

The connection of Krichever's Hitchin description with the previously known linear-algebraic perspective comes about by degenerating \(E\) to its nodal and cuspidal degenerations. A semi-stable bundle on a nodal and cuspidal curve is determined by the descent data from the normalizing \(\PP^1\) \cite{descentdata}. The descent data in the nodal case is given by an element of \(GL_n\) gluing the fibers of the trivial vector bundle over \(\PP^1\) at the two points that map to the nodal point in \(E_{node}\). The moduli of Higgs bundles on the nodal cubic is therefore given by Hamiltonian reduction
\[
T^*\Bun_n(E_{node}) \cong T^* GL_n /\!\!/ GL_n.
\]
Similarly, it is an element of \(\mathfrak{gl}_n\) acting on the fiber over the point mapping to the cuspidal point of \(E_{cusp}\) that defines descent to the cuspidal cubic \(E_{cusp}\). Both of these descent data are defined up to a change of basis, and hence an element of \(GL_n\). Therefore the moduli space of Higgs bundles on \(E\), \(T^* \Bun_n(E)\), is obtained by Hamiltonian reduction in the cuspidal and nodal cases. So the linear algebraic descriptions of the rational and trigonometric CM systems are compatible with the description given by Krichever. 

Finally, the dual Hitchin perspective in terms of Hitchin spectral curves of the CM system was completed by Treibich and Verdier in \cite{Treibich1990}, \cite{TV2} in which the spectral curves of the Hitchin integrable systems were given in terms of tangential covers of the base elliptic curve. These results, among others motivated the authors of \cite{MR2377220} to describe the spin CM hierarchy in terms of spectral sheaves living on the projectivization of the total space of the Atiyah bundle \(\overline{E}^\natural\).

\subsection{Main results} \label{subsec: main results}

The goal of this paper is to provide a description of the full elliptic spin RS system as is given in \cite{MR2377220} for the CM system. The primary object of study is the moduli space \(\mathsf{RS}_{\sigma, n}(E, V)\) from Definition \ref{def: RS phase space}. Let \(S_\sigma = \PP(\mathcal O \oplus \mathcal L_\sigma)\) be the associated \(\PP^1\) bundle to the line bundle \(\mathcal L_\sigma\) given by an element \(\sigma \in \Jac E\). \(S_\sigma\) has two sections, \(E_0\) and \(E_\infty\), and let \(V\) be a torsion sheaf on the divisor \(E_0 \cup E_\infty\). Then the space \(\mathsf{RS}_{\sigma, n}(E, V)\) parametrizes what we call framed RS spectral sheaves. The unframed spectral sheaves are sheaves on \(S_\sigma\) with pure 1-dimensional support on curves forming ramified coverings over \(E\). The additional data of the framing is an isomorphism of the restriction of the sheaf to the two \(\mathscr F|_{E_0 \cup E_\infty} \overset{\sim}{\rightarrow} V\). For simple framing \(V = \OO_{q_0} \oplus \OO_{q_\infty}\) we call the corresponding space the spinless RS phase space, and when \(V = \OO_{q_0}^k \oplus \OO_{q_\infty}^k\) we get the spin system.

The proof that this moduli space actually plays the role of the RS phase space can be split up into 3 main arguments:
\begin{enumerate}
	\item Provide a dual description of \(\mathsf{RS}_{\sigma, n}(E, V)\) in terms of vector bundles on the base curve \(E\) together with a pair of Higgs fields \(\eta_0, \eta_\infty\).
	\item Identify the spin RS Lax matrix as a particular combination of the Higgs fields. 
	\item Define a natural set of flows on \(\mathsf{RS}_{\sigma, n}(E, V)\) which are simultaneously the flows defined by the Hitchin fibration and also isospectral flows of the RS Lax matrix.
\end{enumerate}

In the definition of \(\mathsf{RS}_{\sigma,n}(E,V)\) given above we do not require any conditions for \(V\), but because of the geometry of \(S_\sigma\) the intersection profile of the support of \(\mathscr F\) with \(E_0\) and \(E_\infty\), given by the divisors \(D_0\) and \(D_\infty\) on \(E\) respectively, must be related by translation. This implies the support of \(V_0\) and \(V_\infty\) are shifts, but the underlying sheaves may not be related. As such we restrict to this translated case in the statement of Theorem \ref{thm:higgs description} which resolves step 1 in the above sketch:

\begin{introHiggsthm}
	Let \(V\) be an RS framing sheaf with support \(D_0\) on \(E_0\) and \(D_\infty\) on \(E_\infty\) with the property that \(V_0\) and \(V_\infty\) are translates by \(D_\infty - D_0\). The moduli space \(\mathsf{RS}_{\sigma, n}(E,V)\) parametrizes the the data of a semi-stable rank \(n\) vector bundle \(W\) together with a pair of prolonged Higgs fields
	\begin{align*}	
		\eta_0 : W \lra (W \otimes \mathcal L_\sigma^{-1})(D_\infty) \\
		\eta_\infty: W \lra (W \otimes \mathcal L_\sigma)(D_0)
	\end{align*}
	with framing morphisms \((u_0, v_0)\) and \((u_\infty, v_\infty)\) respectively subject to the condition that \[\eta_0(z) = \eta_\infty^{-1}(z + n(\sigma - b))\]
\end{introHiggsthm}	

Once the Hitchin description is obtained, the relation between \(\eta_0\) and \(\eta_\infty\) closely mirror the description of the matrix of intertwining vectors in \cite{Hasegawa1997} and \cite{math/9909079}. This suggests considering the composition \(\eta_0 \circ \eta_\infty\), which is the precise way in which the RS Lax matrix is obtained in the above references. By translating everything into the language of Weierstrass \(\sigma\)-functions by Lemma \ref{lem: elliptic functions} we resolve part 2 of the sketch by: 
\begin{introLaxprop}
	Let \(\eta_0\) and \(\eta_\infty\) be the associated Higgs fields for an RS spectral sheaf in \(\mathsf{RS}_{\sigma, n}(E, \OO_{p_0} \oplus \OO_{p_\infty})\) where \(E\) is a smooth Weierstrass cubic. On the locus of decomposable vector bundles, the composition \((\eta_0 \otimes id) \circ \eta_\infty\) of equation \eqref{eqn:composition} equals the RS Lax matrix
	\[
	((\eta_0 \otimes id) \circ \eta_\infty)_{k,k'} = \frac{\sigma(z + \hbar + q_k - q_{k'})}{\sigma(z)} \prod_{l \neq k} \frac{\sigma(\hbar + q_l - q_{k'})}{\sigma(q_{l} - q_k)} \exp P_k
	\]
	with \(\hbar = (\sigma - b)\). 
\end{introLaxprop}
The subsequent Corollaries \ref{cor:singular lax} and \ref{cor:spin Lax matrix} resolve the spin RS case, and its singular degenerations. 

Finally we introduce the desired flows of item 3 in the sketch. These are \emph{tweaking} flows on RS spectral sheaves. A simple version of these flows were introduced in \cite{Krichever1980} for the CM system, and they were studied by \cite{PMIHES_2001__94__87_0}, \cite{Donagi1996}, and others in the context of other integrable systems. The terminology was coined in \cite{MR2377220} for their flows on CM spectral sheaves. Informally, tweaking flows can be thought of as changing the gluing isomorphism between the restrictions of the sheaf \(\mathscr F\) to the divisor \(E_0 \cup E_\infty\) and its complement \(S_\sigma^*\). In Proposition \ref{prop: tweaking theorem} we show that the tweaking flows satisfy the requirements of the sketch, and as a corollary we get the main theorem:
\begin{introMainthm}
	The moduli space \(\mathsf{RS}_{\sigma, n}(E, \OO_{q_0}^k \oplus \OO_{q_\infty}^k)\) is isomorphic to a completion of the spin Ruijsenaars-Schneider phase space, with RS flows given by tweaking flows on RS spectral sheaves. 
\end{introMainthm}

\subsection{Motivation and further directions} \label{subsec:motivation}

The impetus for providing this spectral description of the RS system came from trying to elucidate the 2D Toda-RS duality of \cite{MR1379076}. In \cite{MR2835323} the authors reinterpret the KP-CM correspondence originally considered by \cite{doi:10.1002/cpa.3160300106} in terms of a non-commutative Fourier-Mukai transform on \(\mathcal D\)-modules. A key step in their proof of the correspondence is the identification of the CM phase space as the moduli space of spectral sheaves in \(\overline{E}^\natural\). After defining the non-commutative Fourier-Mukai transform based off the previous results of \cite{PMIHES_1987__65__131_0} and \cite{polishchuk2001}, they identify the result of Fourier-Mukai of CM spectral sheaves with \(\mathcal D\)-bundles defining geometric points in the Sato Grassmannian parametrizing solutions to the KP integrable hierarchy.

The authors also set out to provide a similiar argument to prove the 2D Toda-RS duality. The key first step in this project is analogously to find a suitable spectral sheaf description of the RS integrable system. This paper achieves this first goal in Theorem \ref{thm:main theorem}. Mirroring the above argument, the next step should be to define a non-commutative Fourier-Mukai transform for difference modules on \(E\). Na\"ively applying the Fourier-Mukai transform to an RS spectral sheaf \(\mathscr F\), or equivalently its Higgs bundle  \(W\) with \(\mathcal L_\sigma\)-twisted Higgs fields \(\eta_0\), \(\eta_\infty\), the vector bundle \(W\) transforms into a torsion sheaf on \(E\) and the Higgs fields \(\eta_0\) and \(\eta_\infty\) define the steps \(\sigma_+\) and \(\sigma_-\) of a difference module. 

This simple observation suggests that the provided description is the ``correct'' description of the RS phase space for the purposes of providing a geometric description of the 2D Toda - RS correspondence. The final step in this program would be to provide a difference-module description of the phase space of the 2D Toda lattice hierarchy introduced by \cite{ueno1984}. The rest of this project is completed in the forthcoming work \cite{MPTNDBZ}.

\subsection{Structure of the paper} \label{subsec: structure}
We begin in Section \ref{sec:prelim} by introducing the necessary background on elliptic functions that will eventually be used in the proof of Proposition \ref{prop:geometric lax matrix}. With the necessary language of elliptic functions established, we also give explicit formulas and references for various RS Lax matrices that are used in this paper. Section \ref{sec:RS phase space} is where we introduce the main space of interest \(\mathsf{RS}_{\sigma,n}(E,V)\). After briefly discussing the geometry of divisors on \(S_\sigma\) we get to the main goal of providing a Hitchin description of the above moduli space. It is in Section \ref{sec: laxmatrix} where we concentrate on the coordinate description of the Higgs fields obtained in the previous section. After some preliminary matrix calculations we state and prove Proposition \ref{prop:geometric lax matrix} showing that the composition \(\eta_\infty \circ \eta_0\) gives the RS Lax matrix. In Section \ref{sec:flows} we turn to an infinitesmal study of the moduli space \(\mathsf{RS}_{\sigma,n}(E,V)\) where we provide a description of the tangent space in both the spectral and Hitchin descriptions. We also provide the precise definition of the RS tweaking flows in Definition \ref{def:RS spectral flows} which will define the RS hierarchy. Finally, we prove the main theorem, Theorem \ref{thm:main theorem}, which identifies \(\mathsf{RS}_{\sigma,n}(E,V)\) as the RS phase space together with its RS flows being given by tweaking flows on spectral sheaves. We end with Section \ref{sec:comparison} in which we compare our description of the RS system with the CM system. We concentrate on two specific relations: the degeneration of RS to CM in the \(c \rightarrow \infty\) limit, and the duality between the trigonometric CM and rational RS systems. In the first case we introduce the universal RS phase space \(\mathsf{RS}_{n}(E,\mathcal V)\) in \ref{def:universal RS} which provides a unified treatment of the RS system for all \(\sigma \in \Jac E\) and the CM system which can be thought of as the \(\sigma \rightarrow 0\) limit of the RS system. We prove in Proposition \ref{prop: universal RS} that this defines a universal integrable system intertwining all of the above systems. Finally in Section \ref{subsec: TCM-RRS} we try to spell out the trigonometric CM - rational RS duality in the perspective of RS and CM spectral sheaves. 

\section{Preliminaries} \label{sec:prelim}
\subsection{Sections of line bundles on elliptic curves and theta functions} \label{subsec: line bundles}

In this paper we choose to express the basis of sections of line bundles in terms of Weierstrass elliptic functions. These functions have the benefit of behaving well under degeneration to nodal and cuspidal cubic curves, and hence provide a unified way of treating the three forms of the RS system. The identities that are used to obtain the RS Lax matrix written in terms of the so-called intertwining vectors of \cite{MR908997} and \cite{Hasegawa1997}. As such we include a brief dictionary between the two bases of elliptic functions. The results stated in this section on elliptic functions are largely taken from \cite{MR2352717} and \cite{MR1007595}. First, fixing the periods \(\omega_1\) and \(\omega_2\) of the elliptic curve E we define the Weierstrass \(\sigma\)-function defined as the infinite product
\[
\sigma(z|\omega_1, \omega_2) = z \prod_{\substack{w \in \Lambda \\ w \neq 0}}\left(1 - \frac{z}{w}\right)\exp\left(\frac{z}{w} + \frac{z^2}{2 w^2}\right)
\] 
Where \(\Lambda = \mathbb Z \omega_1 \oplus \mathbb Z \omega_2\) is the lattice defining \(E\). When there is no ambiguity about the periods for \(\sigma(z|\omega_1, \omega_2)\) we simply write \(\sigma(z)\). 

The function \(\sigma(z)\) is doubly quasiperiodic with periods \(\omega_1\) and \(\omega_2\):
\begin{align*}
	\sigma(z + \omega_1) &= - e^{2\eta_1(z+\omega_1/2)} \sigma(z) \\
	\sigma(z + \omega_2) &= - e^{2\eta_1(z+\omega_2/2)} \sigma(z)
\end{align*}
for constants \(\eta_1, \eta_2\) related by the equation \(\eta_1 \omega_2 - \eta_2 \omega_1 = i \pi/2\). The key useful feature of \(\sigma(z)\) is that it has simple zeroes at the lattice points of \(\Lambda \subset \mathbb C\), and does not vanish outside this locus. In particular, this classification of the zeroes of \(\sigma\) together with the classical theorem of Appel-Humbert classifying line bundles on elliptic curves (or more general abelian varieties) in terms of factors of automorphy implies that the function 
\[
\Phi_q(z) = \frac{\sigma(z - q)}{\sigma(z)}
\]
defines a section of the line bundle on \(E\) corresponding to the divisor \(q-0\). After fixing a basepoint \(b \in E\) we can identify this line bundle with \(\OO_E(q - b)\).
The Weierstrass \(\sigma\)-function is related to the Weierstrass \(\wp\)-function by the identity
\[\wp(z) = -\frac{d^2}{dz^2} \log \sigma(z)\]
where we similarly supress the dependence of \(\wp\) on the periods of \(E\). In the limit of \(\omega_2 \rightarrow \infty\) the functions \(\sigma(z)\) and \(\wp(z)\) degenerate to the trigonometric functions
\begin{align*}
	\sigma(z) &\mapsto \sin(z) \\
	\wp(z) &\mapsto 1/\sin(z)^2
\end{align*}
up to unimportant constant factors. Further degenerating \(\omega_1, \omega_2 \rightarrow \infty\) such that the ratio \(\omega_1 / \omega_2\) is fixed, the Weierstrass elliptic functions degenerate to 
\begin{align*}
	\sigma(z) &\mapsto z \\
	\wp(z) &\mapsto 1/z^2
\end{align*}
again up to some factors. This nice degenerating behavior allows for the previously mentioned uniform treatment of the rational, trigonometric, and elliptic cases of the RS system. 

We next turn to another set of elliptic functions. Let \(\tau = \omega_i/\omega_j\) be the ratio of the \(\omega_i\) with positive imaginary part. The lattice \(\mathbb Z \oplus \mathbb Z \tau\), and the lattice \(\Lambda\) from above define the same elliptic curve \(E\). The infinite sum
\[
\thetachar{a}{b}(z | \tau) = \sum_{k \in \mathbb Z} \exp(\pi i \tau(k+a)^2 + 2 \pi i (k+a)(z+b))
\]
is called the theta function with characteristics \((a,b)\). When there is no ambiguity we again suppress the dependence on \(\tau\), \(\thetachar{a}{b}(z) = \thetachar{a}{b}(z|\tau)\). These theta functions are also quasi-periodic with respect to the lattice spanned by \(1, \tau\):
\begin{align*}
	\thetachar{a}{b}(z + 1) &= e^{2 \pi i a} \thetachar{a}{b}(z) \\
	\thetachar{a}{b}(z + \tau) & = e^{-\pi i(2z + 2b + \tau)} \thetachar{a}{b}(z). 
\end{align*}
And similar to the Weierstrass \(\sigma\)-function we have that the functions \(\thetachar{a}{b}(z)\) are entire functions with simple zeroes only at the lattice point translates of 
\[(a + 1/2) \tau + (b + 1/2).\]
In particular for \(a = b = 1/2\) the zeroes of \(\theta_1(z) = \thetachar{1/2}{1/2}(z)\) match those of \(\sigma(z)\). Therefore the ratio \(\theta_1(z)/\sigma(z)\) will be a non-vanishing doubly quasiperiodic function. These are called trivial theta functions and they have the form 
\(\vartheta(z) = C e^{A z^2 + B z}\) for some constants \(A, B, C \in \mathbb C\). This argument extends to general theta functions with characteristics where we have the relation 
\begin{align}\label{eqn:sigma shift}
	\sigma(z - a \tau - b) = C_{a,b}e^{A_{a,b} z + B_{a,b}z^2}\thetachar{\frac{1}{2}+a}{b}(z).
\end{align}
By comparing the known values of \(\sigma(z)\) and \(\thetachar{a}{b}(z)\) it is possible to pin down the precise values of the constants \(A_{a,b},B_{a,b},C_{a,b}\), but for our purposes the values can be left unspecified. 

We are primarily interested in using this relation between the Weierstrass \(\sigma\)-function and theta functions with characteristics in the specific case of relating the intertwining vectors mentioned above. 

Let \(\varepsilon_k\) for \(k = 1 \ldots n\) be the standard basis for \(\mathbb C^n\), and let \(\langle \mathbf v , \mathbf w \rangle\) denote the standard inner product on \(\mathbb C^n\) for which \(\langle \varepsilon_i, \varepsilon_j \rangle = \delta_{i,j}\). Denote \(\bar{\varepsilon}_k = \varepsilon_k - \frac{1}{n}\sum_{i=1}^n \varepsilon_i\) the projection of the basis vectors onto the subspace of \(\mathbb C^n\) orthogonal to the diagonal subspace \(\mathbb C \overset{\Delta}{\hookrightarrow} \mathbb C^n\). Finally, for \(\lambda \in \mathbb C^n\), and \(j \in \mathbb Z / n \mathbb Z\) define the intertwining vector
\begin{align}\label{eqn:intertwining def}
	\phi_{\lambda, j}^{\lambda + \eta \bar{\varepsilon}_k}(z) = \thetachar{\frac{1}{2n}-\frac{j}{n^2}}{0}\left(z - n \langle \lambda, \bar{\varepsilon}_k \rangle + \frac{1}{2} \tau\right)
\end{align}

By using equation \eqref{eqn:sigma shift} to transform the definition of the intertwining vectors into the basis of Weierstrass elliptic functions we obtain that:

\begin{lem}\label{lem: elliptic functions}
	Given the above established notation:
	\[
	\phi_{\lambda, j}^{\lambda + \eta \bar{\varepsilon}_k}(z) =\vartheta_j(z) \sigma\left(z - n \langle \lambda, \bar{\epsilon}_k\rangle + \frac{j}{n}\tau\right)
	\]
	where \(\vartheta_j(z)\) is a trivial theta function depending only on \(j\) of the form:
	\[
	\vartheta_j(z) = \frac{1}{\pi \theta_1'(0)}\exp\left(-2 \pi i \left(\frac{-j}{{n}}\right)\left(z+ \frac{1}{2} + \frac{-j \tau}{4 n}\right) + \frac{\eta_1}{2}z^2\right)
	\]
	where \(\eta_1 = -\sigma'(1)/\sigma(1)\). 
\end{lem}

\subsection{Recollections on the Ruijsenaars Schneider system} \label{subsec: recollections on RS}
We now quickly recall the various Lax matrices of the RS system. The first is the matrix introduced by Ruijsenaars in his proof of integrability of the RS system in \cite{MR851627}. The original proof of integrability came from explicitly writing the integrals of motion without giving a matrix generating them. It was only in the second proof of integrability in loc. cit. that Lax matrix was introduced that generates the integrals of motion. Indeed any matrix with the correct eigenvalues could serve as a candidate for the Lax matrix, but the natural solution that comes to the authors is 
\[
(L_{RS}')_{i,j} = e^{\theta_i} \prod_{l \neq i} f(q_i - q_l) \frac{\sigma(q_i - q_j + \lambda)\sigma(\mu)}{\sigma(\lambda)\sigma(q_i - q_j + \mu)}. 
\]
Where we now define the terms appearing in the above product. First \(\lambda\) is the spectral paraemeter, and \(\mu\) is an additional parameter appearing in the definition of the RS system which in our case will be related to the choice of \(\sigma \in \Jac E\), and finally
\[
f^2(q) = \sigma(\mu)^2(\wp(\mu) - \wp(q)).
\]

This first RS Lax matrix is not ideal for the geometric perspective of this paper. Specifically, the square root that appears in the definition suggests that the Lax matrix is defined on a double cover of \(E\). Though as described in Remark \ref{rmk: alt Lax matrix} and \cite{Vasilyev:2018cyt} it is possible to absorb the roots into a redefinition of \(\theta_i\), an alternative Lax matrix for the RS system introduced in \cite{Hasegawa1997} will naturally appear in the geometric perspective of this paper. With this definition of \(L_{RS}'\) we get that 
\[
H_{RS} = \Tr( L_{RS}' + {L_{RS}'}^{-1})
\]
with \(H_{RS}\) the RS Hamiltonian defined in the introduction. 

The origin of this second Lax matrix came from the study of the quantum RS system defined in terms of commuting difference operators. The classical parts of these operators are generated as symmetric functions of a single matrix \(L'_{RS}\) that was noticed to arise in a geometric perspective by \cite{math/9909079} as coming from elementary modifications of vector bundles on an elliptic curve. Explicitly the matrix \(L_{RS}\) is given by
\[
(L_{RS})_{i,j} = \frac{\theta_1(\mu + \lambda + q_j - q_i)}{\theta_1(\lambda)}\prod_{k \neq j} \frac{\theta_1(\mu + q_k - q_i)}{\theta_1(q_k - q_j)}\exp{P_i} 
\]

By the remarks in the previous section we can easily translate this formula for the Lax matrix into one that is defined in terms of the Weierstrass \(\sigma\)-function by introducing factors of trivial theta-functions. 

The Lax matrix \(L_{RS}\) is the natural matrix that appears in the study of the RS system in terms of the RS spectral sheaves of this paper, so it will be the form of the Lax matrix of primary interest in this paper. Hasegawa introduces the matrix \(\Phi(\lambda)\) and proves that up to a constant, one can recover the original Ruijsenaars Lax matrix from \(L_{RS}\) from conjugation by \(\Phi(\lambda)^{1/2}\). We do not have a geometric description of \(\Phi(\lambda)\), so this connection with the Ruijsenaars Lax matrix is not entirely geometrically understood. On the other hand, a relation between \(L'_{RS}\) and \(L_{RS}\) is recovered in a different way as describe in Remark \ref{rmk: alt Lax matrix}.

Finally, the last Lax matrix we will consider is that of Krichever defined in the study of the 2D Toda-RS duality \cite{MR1379076}. 
\[
(L''_{RS})_{i,j} = \frac{\sigma(\lambda + q_i - q_j)}{\sigma(\lambda + \mu)\sigma(q_i - q_j - \mu)}\left[\frac{\sigma(z - \mu)}{\sigma(z + \mu)} \right]^{(q_i - q_j - \mu)/(2 \mu)}
\]
This Lax matrix is also related to the Hasegawa Lax matrix as described in section 5.2 of \cite{Hasegawa1997}, this time as a coefficient of \(\hbar^1\) in the \(\hbar \rightarrow 0\) limit of \(\tilde{L}'_{RS}\). Where now \(\tilde{L}'_{RS}\) is given explicitly as a conjugation of \(L'_{RS}\) by the matrix of intertwining vectors. We will not consider \(L''_{RS}\) in this paper, but we provide its definition for completeness.

\section{The RS phase space} \label{sec:RS phase space}
In this section we introduce the main object of study, the phase space of the completed spin Ruijsenaars-Schneider system. Before we get to the full spin system, we begin with a preliminary definition which will help in understanding the geometry of the RS system. Throughout let $E$ be a Weierstrass cubic curve. Recall that Weierstrass cubics come in three flavors, all of which are isomorphic to one of the following three curves in $\PP^2$:
\begin{itemize}
	\item Smooth, defined by the equation \(zy^2 = x(x-1)(x-\lambda)\) for \(\lambda \neq 0\) or \(1\).
	\item Nodal, defined by the equation \(z y^2 = x^2(x-1)\).
	\item Cuspidal, defined by the equation \(z y^2 = x^3 \). 
\end{itemize}

We will refer to these three cases as the \emph{elliptic}, \emph{trigonometric}, and \emph{rational} cases respectively. The terminology comes from considering the ring of functions on the universal cover \(\mathbb C\) which descend to \(E\).

The smooth locus of \(E\) also inherits the structure of an algebraic group after fixing a point $b \in E^{sm}$. For \(E\) smooth, the points of \(E\) have the usual structure of the 1-dimensional algebraic group corresponding to the elliptic curve \(E\). For \(E\) nodal, \(E^{sm} \cong \mathbb C^\times\) as an algebraic group under multiplication, and for \(E\) cuspidal, \(E^{sm} \cong \mathbb C\) as an algebraic group under addition. 

Denote \(\Jac E\) the Jacobian of \(E\). In all three of the above cases, \(\Jac E \cong E^{sm}\) as algebraic groups. Furthermore let \(\overline{\Jac} \, E\) denote the compactified Jacobian. This is the moduli space of rank 1 torsion free sheaves of degree 0 on \(E\). In all three cases, we also have that \(E \cong \overline{\Jac} \, E\). Let \(\sigma \in \Jac E \cong E\) correspond to the line bundle \(\mathcal L_\sigma\), and let \(S_\sigma = \PP(\OO \oplus \mathcal L_\sigma)\overset{p}{\rightarrow}E\) be the associated \(\PP^1\) bundle over \(E\). By a curve in \(S_\sigma\) we mean a Cartier subscheme of \(S_\sigma\).

\begin{defn} \label{def:RS spectral curve}
	An RS spectral curve is a complete curve \(\Sigma \subset S_\sigma\) for which the restriction of the projection \(p |_\Sigma:\Sigma \lra E\) is a finite covering map.
\end{defn}


The surface \(S_\sigma\) has two sections \(E_0\) and \(E_\infty\) corresponding to the two projections \(\OO \oplus \mathcal L_\sigma \rightarrow \OO\) and \(\OO \oplus \mathcal L_\sigma \rightarrow \mathcal L_\sigma\) respectively. This yields the decompositions of \(S_\sigma\):
\[
S_\sigma = S_\sigma^0 \cup E_\infty =  E_0 \cup S_\sigma^\infty = E_0 \cup S_\sigma^* \cup E_\infty
\]
into affine bundles \(S_\sigma^0\), \(S_\sigma^\infty\) and a \(\mathbb C^* \) bundle \(S_\sigma^*\). RS specral curves being complete cannot lie entirely in the affine bundles \(S_\sigma^0\) and \(S_\sigma^\infty\) and hence will intersect the sections \(E_0\) and \(E_\infty\) at two divisors \(D_0\) and \(D_\infty\) on \(E\). The geometry of the surface \(S_\sigma\) guarantees a relation between the two intersection divisors.

\begin{notn}
	For a divisor \(D = \sum_q n_q q\) in \(E\), denote \(F_D = \sum_q n_q F_q\), where \(F_q\) is the divisor of the fiber above \(q\in E\). In particular, \(p^* \OO_E(D) \cong \OO_{S_\sigma}(F_D)\). 
\end{notn}

\begin{lem} \label{lem:divisor relation}
	For an RS spectral curve \(\Sigma\) for which the projection \(p|_\Sigma\) is generically \(n\) to \(1\), the divisor classes of intersections \(D_0\) and \(D_\infty\) are related by the equation
	\begin{equation}\label{eqn:divisor relation}
			D_0 = D_\infty + n (\sigma - b).
	\end{equation}
\end{lem}

\begin{proof}
Before we can relate the divisors \(D_0\) and \(D_\infty\) we first recall the geometry of the divisor class group of \(S_\sigma\). 

Because \(S_\sigma\) is a \(\PP^1\) bundle on E, it is known that the divisor class group of \(S_\sigma\) is isomorphic to \(\mathrm{Cl}(E) \oplus \mathbb Z [\lambda]\) where \(\lambda\) is a class of a section of \(S_\sigma\). In particular, the divisor classes \(E_0\) and \(E_\infty\) cannot be independent and are related by \[E_0 = E_\infty + F_D\] for some divisor \(D \in \mathrm{Cl}(E)\). 

The divisor \(D\) can be recovered by considering the self-intersection \[E_0.E_0 = E_0.E_\infty + E_0.F_D = 0 + E_0.F_D.\] This equation implies that the divisor \(D\) is simply the class of the normal bundle of \(E_0\) in \(S_\sigma\). Being a zero-section of the vector bundle \(S_\sigma^0 = \mathrm{Tot}(\mathcal L_\sigma)\) implies that the normal bundle \(\mathcal N_{E_0/S_\sigma}\) is equal to the line bundle \(\mathcal L_\sigma\), which corresponds to the divisor class \((\sigma - b)\) after a choice of basepoint \(b \in E\) is chosen. Hence we obtain that 
\begin{align}\label{eqn:sections relation}
	E_0 = E_\infty + (F_\sigma - F_b).
\end{align}

Returning to the intersection divisors \(D_0\) and \(D_\infty\), an RS spectral curve \(\Sigma\) satisfying the assumptions of the lemma will lie in the divisor class group \(|n E_\infty+ F_{D_0}|\), and we can calculate the intersection with \(E_\infty\) using the relation \eqref{eqn:divisor relation}:
\[D_\infty = (n E_\infty+ F_{D_0}).E_\infty = n E_\infty . E_\infty + D_0 = n(b - \sigma) + D_0\]
which yields the desired result.

\end{proof}

In order to treat the general spin Ruijsenaars-Schneider system we will consider \emph{framed spectral sheaves} on \(S_\sigma\). These will be sheaves supported on RS spectral curves together with a framing at \(E_0\) and \(E_\infty\). 

\begin{defn}\label{def: framing}
	Let \(V\) be a sheaf on \(S_\sigma\) supported scheme-theoretically on the divisor \(E_0 \cup E_\infty\) which is torsion and supported on the smooth locus \(E_0^{sm} \cup E_\infty^{sm}\). For a sheaf \(\mathscr F\) supported on an RS spectral curve, a \emph{framing isomorphism} is an isomorphism:
	\[
	\varphi: \mathscr F|_{E_0 \cup E_\infty} \overset{\sim}{\lra} V.
	\] 
	We may refer to the sheaf \(V\) without reference to \(\varphi\) as simply the \emph{framing} of \(\mathscr F\)
\end{defn}

The support of V being contained on a disjoint union of sections, allows us to decompose \(V\) into the direct sum \(V = V_0 \oplus V_\infty\) with \(V_0\) and \(V_\infty\) torsion sheaves supported on \(E_0\) and \(E_\infty\) respectively. By Lemma \ref{lem:divisor relation}, the divisor classes of the support of \(V_0\) and \(V_\infty\) are constrained to be shifted by the factor \(n(b - \sigma)\) for a sheaf supported on an n-fold cover RS spectral curve.

We now turn to the definition for what will be the phase space for the RS system.

\begin{defn}\label{def: RS phase space}
	For a fixed \(\sigma \in \Jac E\), \(n \in \mathbb N\), and framing sheaf \(V\) let \(\mathsf{RS}_{\sigma,n}(E,V)\) be the moduli space of sheaves \(\mathscr F\) on \(S_\sigma\) sastisfying:
	\begin{itemize}
		\item \(\mathscr F\) is a sheaf of pure dimension 1, supported on an RS spectral curve \(\Sigma\).
		\item \(\mathscr F\) admits a framing \(\varphi:\mathscr F|_{E_0 \cup E_\infty} \overset{\sim}{\lra} V\).
		\item The sheaves \(W_k = p_* \mathscr F(-kE_\infty)\) are rank \(n\), degree \((k+1) \deg V_\infty\), vector bundles on \(E\) for all \(k \in \mathbb Z\). Furthermore, \(W_{-1}\) is semi-stable.
		\item If \(E\) is singular, the pullback of \(W_{-1}\) to the normalization is a trivial vector bundle.
	\end{itemize}
	Such sheaves are called \emph{framed RS spectral sheaves}.
\end{defn}

We call framed RS spectral sheaves with framing \(V = \OO_{p_0} \oplus \OO_{p_\infty}\) \emph{spinless} RS spectral sheaves, and when \(V = \OO_{p_0}^{k} \oplus \OO_{p_\infty}^k\) we call them \emph{spin} RS spectral sheaves. 

\begin{rmk}\label{rmk:framing}
	Note that by Lemma \ref{lem:divisor relation} the varieties \(\mathsf{RS}_{\sigma, n}(E, V)\) are empty unless the sheaves \(V_0\) and \(V_\infty\) have the same length, and their divisors of support are related by equation \eqref{eqn:divisor relation}. In both the spin and spinless cases the sheaves \(V_0\) and \(V_\infty\) are simply translates of each other by \(n(\sigma - b)\), though there are examples of spectral sheaves with framings that do not satisfy this translation condition.

	For example consider a curve \(\widetilde{\Sigma}\) in the linear series \(|2E_0 + 2 F_b|\) which has a degree 2 transverse intersection at a point \(q_\infty\) along \(E_\infty\) and an order 2 tangential intersection at a point \(q_0\) along \(E_0\) where \(q_\infty\) and \(q_0\) are the appropriate shift of each other. The dimension of the linear series \(|2 E_0 + 2 F_b|\) has dimension \(\geq 4\), and hence such a curve exists. 

	In a sufficiently small open set near \(q_\infty \in E_\infty\), with the divisor \(E_\infty\) defined by \(y = 0\), the structure sheaf of the curve is given by \(\mathbb C[x, y]/((x+y)(x-y))\) which restricts to \(E_\infty\) to give the \(\mathbb C[x]/(x^2)\). Similarly the structure sheaf near the tangential intersection will also restrict to give the non-split length 2 torsion sheaf. On the other hand, if \(p:\widehat{\Sigma} \rightarrow \widetilde{\Sigma}\) is the normalization of \(\widetilde{\Sigma}\) then \(\mathcal G = p_*(\OO_{\widehat{\Sigma}})\) will be a sheaf on \(\widetilde{\Sigma}\) which is isomorphic away from the point \(q_\infty\). Therefore we can conclude that \(\mathcal G|_{q_0}\) is still isomorphic to the non-split extension.

	On the other hand at \(q_\infty\), we claim \(\mathcal G|_{q_\infty}\) is isomorphic to the split length 2 torsion sheaf supported at \(q_\infty\). In fact \(\OO_{\widehat{\Sigma}}\) restricted to each of the two points \(q_1, 1_2\) mapping to \(q_\infty\) on \(\widetilde{\Sigma}\) will give a 1 dimensional vector space with a trivial \(\OO_{E_\infty, q_\infty}\) action, and hence the full pushforward of \(\mathcal G|_{q_\infty}\) will decompose as a direct sum. 



	Most of the theorems for the framed RS system will be unaffected by this complication, but when the assumption that \(V_0\) and \(V_\infty\) are shifts is necessary we will explicitly state it. 
\end{rmk}

\subsection{From RS spectral sheaves to Higgs bundles}
To connect the moduli space \(\mathsf{RS}_{\sigma, n}(E,V)\) with RS phase spaces, we first turn to geometrically constructing the RS Lax matrix from RS spectral sheaves. The Lax matrix will appear as meromorphic twisted Higgs fields on the bundles \(W_k\) obtained from \(\mathsf{RS}_{\sigma, n}(E, V)\). Before we tackle the Higgs field description of \(\mathsf{RS}_{\sigma, n}(E, V)\), we recall the classical story for \(GL_n\)-Higgs bundles on a complex curve \(\Sigma\) and their associated spectral curves.
\subsubsection*{\(GL_n(\mathbb C)\)-Higgs bundles}
The moduli space \(\mathcal M_{Higgs}(\Sigma)\) of Higgs bundles on a complex curve \(\Sigma\) parametrizes pairs \((\mathcal E, \phi)\) of a vector bundle \(\mathcal E\) together with a morphism \(\mathcal E \overset{\phi}{\rightarrow} \mathcal E \otimes K \). There are Koszul dual coordinates on \(\mathcal M_{Higgs}(\Sigma)\) given by pairs \((\widetilde{\Sigma}, \mathscr L)\) of a curve \(\widetilde{\Sigma} \subset T^* \Sigma\) for which the restriction of the projection map is a ramified covering, and a line bundle \(\mathscr L \) on \( \widetilde{\Sigma}\). 

Given a spectral curve and line bundle \((\widetilde{\Sigma}, \mathscr L)\), this classical Koszul duality is obtained by setting the vector bundle to be \(\mathcal E = \pi_* \mathscr L\) where \(\pi: \widetilde{\Sigma} \lra \Sigma\) is the restriction of the projection from \(T^* \Sigma\) to \(\widetilde{\Sigma}\), and the morphism \(\phi:\mathcal E \lra \mathcal E \otimes K\) comes from the action of \(\OO_{T^*\Sigma}\) on \(\mathscr L\). In particular, the vector bundle \(\OO_{T^*\Sigma, \leq 1}\) of sections of degree \(\leq 1\) in the fibers of \(T^*\Sigma\) pushes forward to \(K^\vee\). Therefore the action 
\[
\OO_{T^* \Sigma, \leq 1} \otimes \mathscr L \lra \mathscr L
\]
pushes forward to the morphism of vector bundles on \(\Sigma\):
\[
K^\vee \otimes \mathcal E \lra \mathcal E
\]
which we identify with \(\phi\). Conversely, the data of a pair \((\mathcal E, \phi)\) is enough to reconstruct the spectral curve together with a line bundle. If \(x \in H^0(T^* \Sigma, \pi^* K)\) is the unique tautological section linear in the fibers of \(T^*\Sigma\), then \(\widetilde{\Sigma}\) is given as the locus of vanishing of the polynomial \(\det(x Id - \pi^* \phi)\). The line bundle is furthermore obtained by considering \(\coker \pi^* \phi \), which will be supported on \(\widetilde{\Sigma} \) and generically rank 1. 

The moduli space \(\mathcal M_{Higgs}(\Sigma)\) is a complex symplectic variety, and the Koszul dual coordinates \((\mathcal E, \phi)\) and \((\widetilde{\Sigma}, \mathcal L)\) are both symplectic coordinates. The coordinates\((\widetilde{\Sigma}, \mathscr L)\) can be seen as the action-angle coordinates for the natural Hitchin integrable system that exists on \(\mathcal M_{Higgs}(\Sigma)\). Specifically, let
\[
\det(x Id - \pi^* \phi) = x^n - H_1(\phi) x^{n-1} + \ldots (-1)^{n-1} H_{n-1}(\phi) x + (-1)^n H_n(\phi)
\]
for sections \(H_i \in H^0(\Sigma, K^{\otimes i})\), and define the Hitchin base
\[
\mathcal B_{Higgs}(\Sigma) = \bigoplus_{i=0}^{n-1} H^0(\Sigma, K^{\otimes i}).
\]
Then \(\mathcal M_{Higgs}(\Sigma)\) has an natural map to \(\mathcal B_{Higgs}(\Sigma)\) given by
\[
\mathbf{H}(\mathcal E, \phi) = (H_1(\phi), \ldots,H_{n-1}(\phi)).
\]
The fibers of \(\mathbf{H}\) are the Jacobians of the spectral curves \(\widetilde \Sigma\), realizing \(\mathcal M_{Higgs}(\Sigma) \overset{\mathbf H}{\rightarrow} \mathcal B_{Higgs}(\Sigma)\) as an algebraically completely integrable Hamiltonian system. The coordinates \((\widetilde \Sigma, \mathscr L)\) are then the action-angle coordinates for this integrable system. 
\begin{rmk}
	The reason for describing the pair of coordinates as \emph{Koszul dual} is that for higher dimensional varieties \(X\) the map \(\mathcal E \overset{\phi}{\rightarrow} \mathcal E \otimes K\) extends to a complex
	\[
	\mathcal E \overset{\phi}{\lra} \mathcal E \otimes K \overset{\wedge \phi}{\lra} \mathcal E \otimes \bigwedge^2 K \overset{\wedge \phi}{\lra} \ldots \overset{\wedge \phi}{\lra} \mathcal E \otimes \bigwedge^{\dim X} K
	\]
	realizing \(\mathcal E\) as a graded \(\bigwedge^\bullet K\)-module. 

	Dually, by living in \(T^* \Sigma\), \(\mathscr L\) is naturally a graded module for the symmetric algebra \(S^\bullet K \). The transformation between the pairs of coordinates \((\mathcal E, \phi)\) and \((\widetilde \Sigma, \mathscr L)\) can be realized as the classical Koszul duality for modules between the symmetric and alternating algebras. This perspective on Higgs bundles is expanded in more detail in \cite{duality}.
\end{rmk}

\subsubsection*{Hitchin description of moduli of RS-spectral sheaves} Taking this classical picture as our guide, the definition of \(\mathsf{RS}_{\sigma, n}(E, V)\) is analogous to the coordinates given by spectral curves and line bundles on the moduli space of Higgs bundles. In particular, there exists a locus of sheaves in \(\mathsf{RS}_{\sigma, n}(E, \OO_{p_0} \oplus \OO_{p_\infty})\) that are of the form \(i_* \mathscr L\) where \(i: \widetilde \Sigma \hookrightarrow S_\sigma\) is the closed embedding of and RS spectral sheaf \(\widetilde \Sigma\), and \(\mathcal L\) is a line bundle on \(\widetilde{\Sigma}\). 

Even in the more general setting for a pure 1-dimensional sheaf \(\mathscr F\) on \( S_\sigma\), there exists a Koszul dual description:

\begin{prop} \label{prop:Koszul duality}
	The points of the moduli space \(\mathsf{RS}_{\sigma, n}(E, V)\) are in bijection with the following data:
	\begin{enumerate}
		\item A pair of semi-stable rank \(n\), degree \(0\) vector bundles \(W\) and \(W'\) on \(E\).
		\item A pair of short exact sequences 
		\begin{align}\label{eqn:pair of ses}
		0 \lra W \overset{\iota_0}{\lra}  W' \lra V_\infty \lra 0 \\
		0 \lra W \otimes \mathcal L_\sigma^{-1} \overset{\iota_\infty}{\lra}  W' \lra V_0 \lra 0	
		\end{align}
	\end{enumerate}
	Subject to the relation that the morphisms \(\iota_0 \circ \iota_\infty^{-1} : W \lra W \otimes \mathcal L_\sigma^{-1}\) and \(\iota_\infty \circ \iota_0^{-1}: W \otimes \mathcal L_\sigma^{-1} \lra W\) are inverses on the open set \(U = E \setminus (D_0 \cup D_\infty)\).
\end{prop}

\begin{proof}
	Beginning with an RS spectral sheaf \(\mathscr F\), we can forget about the sections \(E_\infty\) and \(E_0\) and restrict \(\mathscr F\) to the two affine bundles \(S_\sigma^0\) and \(S_\sigma^\infty\). On each of these affine bundles the Koszul duality of \cite{duality} implies we have a bijection between \(\mathscr F|_{S_\sigma^0}\) and the collection of data:
	\begin{enumerate}[(i)]
		\item A pair of vector bundles \(W = W_1 = p_*\mathscr F(-E_\infty)\) and \(W_0 = W' = p_* \mathscr F\).
		\item The short exact sequence of sheaves on \(E\) 
		\begin{align}\label{eqn:infty ses}
			0 \lra W \overset{\iota_0}{\lra} W' \lra V_\infty \lra 0
		\end{align}
		obtained from tensoring the short exact sequence on \(S_\sigma\)
		\[0 \lra \OO_{S_\sigma}(-E_\infty) \lra \OO_{S_\sigma} \lra \OO_{E_\infty} \lra 0\]
		with \(\mathscr F\) and pushing forward to \(E\). 
		\item An action map extending \(\iota_0\), \(a_0: (\OO_E \oplus \mathcal L_\sigma^{-1}) \otimes W \lra W'\) coming from pushing forward the action of a sub-bundle of the structure sheaf \(\OO_{S_\sigma}(E_\infty)_{\leq 1} \otimes \mathscr F(-E_\infty) \lra \mathscr F\). 
	\end{enumerate}
	Whereas restricting to \(S_\sigma^\infty\) we can appeal to the same argument and obtain the bundle \(\widetilde{W} = p_* \mathscr F(-E_0)\) and a short exact sequence 
	\begin{align}\label{eqn:zero ses}
		0 \lra \widetilde{W} \overset{\iota_\infty}{\lra} W' \lra V_0 \lra 0.
	\end{align} 
	But by equation \eqref{eqn:sections relation} and the projection formula we see that \(\widetilde{W} = W \otimes L_\sigma^{-1}\). Furthermore the affine bundle \(S_\sigma^\infty = \Tot L_\sigma^{-1}\), so the corresponding action becomes a map
	\[a_\infty: (\mathcal L_\sigma \oplus \OO_E) \otimes \widetilde{W} \cong (\OO_E \oplus \mathcal L_\sigma^{-1}) \otimes W \lra W'\]
	extending \(\iota_\infty\) on the \(\mathcal L_\sigma^{-1} \otimes W\) summand. 

	Under the identification of \(W \otimes \mathcal L_{\sigma}^{-1}\) with \(\widetilde{W}\), the action \(a_0\) on the direct summand \(\mathcal L_\sigma^{-1} \otimes W\), matches up with the map \(\iota_\infty\) in the short exact sequence \eqref{eqn:zero ses}. Similarly, the action map \(a_\infty\) on the direct summand \(W\) matches up with the map \(\iota_0\) in the short exact sequence \eqref{eqn:infty ses}. Hence the action maps \(a_0\) and \(a_\infty\) are related by a transposition of factors, and are determined by the two short exact sequences. 

	Finally, the two local sections in the \(\mathbb P^1\) fibers of \(S_\sigma\) commute, hence the compositions
	\[
	\mathscr F(- E_\infty) \overset{x_0}{\lra} \mathscr F \overset{x_\infty}{\lra} \mathscr F(E_0)
	\]
	and
	\[
	\mathscr F(- E_\infty) \overset{x_\infty}{\lra} \mathscr F(- E_\infty + E_0) \overset{x_0}{\lra} \mathscr F(E_0)
	\]
	are equal as maps from \(\mathscr F(-E_\infty)\) to \(\mathscr F(E_0)\). These compositions push forward to the two maps
	\[
	\iota_{0,\infty},\iota_{\infty,0}: W \lra W_{-1}\otimes \mathcal L_\sigma
	\]
	which also equal. For our purposes, we choose to repackage the above relation to one on the open set \(U\) on which \(\iota_0\) and \(\iota_\infty\) are invertible. On \(U\) we get that the two maps \(\iota_0 \circ \iota_\infty^{-1} : W \lra W \otimes \mathcal L_\sigma^{-1}\) and \(\iota_\infty \circ \iota_0^{-1}: W \otimes \mathcal L_\sigma^{-1} \lra W\) are inverses.

	Conversely, to obtain an RS spectral sheaf from the pair of short exact sequences, the pair of maps \(\iota_0\) and \(\iota_\infty\) are used to define a pair of action maps \(a_0, a_\infty: (\OO_E \oplus \mathcal L_\sigma^{-1})\otimes W \lra W'\) which commute by the assumptions on \(\iota_0\) and \(\iota_\infty\). The maps \(a_0, a_\infty\) define an action on the pullback of \(W\) from which we can define an RS spectral sheaf \(\mathscr F\) in \(S_\sigma\) with framings at \(E_0\) and \(E_\infty\) by considering the cokernel sheaf of the lift to \(S_\sigma\).
\end{proof}

The action map \(a_0\) or \(a_\infty\) will yield the Higgs field in the Hitchin description of the moduli space \(\mathsf{RS}_{\sigma, n}(E, V)\). The remaining step in obtaining this description follows the argument in \cite{MR2377220}. The end result yields the following description of the moduli space in terms of a pair of prolonged Higgs bundles. We now recall the definition from \cite{duality} and \cite{MR2377220}:

\begin{defn}
	On a scheme \(X\) fix a line bundle \(\mathcal L\), and a torsion sheaf \(V\) supported on a divisor \(D\). A prolonged \(\mathcal L\)-twisted Higgs bundle with framing \(V\) is the tuple \((W, \eta, u, v)\) where:
	\begin{enumerate}
		\item \(W\) is a vector bundle on \(X\).
		\item \(\eta: W \lra (W \otimes \mathcal L)(D)\) is a meromorphic \(\mathcal L\)-twisted Higgs bundle.
		\item \(u: V \lra W(D)/W\) and \(v: W(D)/W \lra V\) are morphisms of torsion sheaves.
	\end{enumerate}
	Subject to the relation on the principal part along \(D\) of the Higgs field \(p.p. \,\, \eta: W/W(-D) \lra W(D)/W\) is related to \(u\) and \(v\) by:
	\[
	p.p. \,\, \eta = u \circ v
	\]
\end{defn}

\begin{thm}\label{thm:higgs description}
	Let \(V\) be an RS framing sheaf with support \(D_0\) on \(E_0\) and \(D_\infty\) on \(E_\infty\) with the property that \(V_0\) and \(V_\infty\) are translates by \(D_\infty - D_0\). The moduli space \(\mathsf{RS}_{\sigma, n}(E,V)\) parametrizes the the data of a semi-stable rank \(n\) vector bundle \(W\) together with a pair of prolonged Higgs fields
	\begin{align*}	
		\eta_0 : W \lra (W \otimes \mathcal L_\sigma^{-1})(D_\infty) \\
		\eta_\infty: W \lra (W \otimes \mathcal L_\sigma)(D_0)
	\end{align*}
	with framing morphisms \((u_0, v_0)\) and \((u_\infty, v_\infty)\) respectively subject to the condition that \[\eta_0(z) = \eta_\infty^{-1}(z + n(\sigma - b)).\]
\end{thm}

\begin{proof}
	By Proposition \ref{prop:Koszul duality}, the theorem follows from showing that the data of the pair of short exact sequences in the proposition is equivalent to the pair of prolonged Higgs fields in the theorem. First note that the vector bundles \(W\) and \(W'\) are isomorphic on the open locus \(E \setminus D_\infty\). Therefore we can define a map \[W'(-D_\infty) \overset{\tilde{u}}{\lra} W\] by recognizing the sections of \(W'\) that vanish along \(D_\infty\) as sections of \(W\). Equivalently, we get a map \(\tilde{u}: W' \lra W(D_\infty)\) which restricts to the standard inclusion \(W \hookrightarrow W(D_\infty)\) on the image of \(W\) in \(W'\) via the map \(\iota_\infty\). 

	Following the above discussion for both short exact sequences in the data of Proposition \ref{prop:Koszul duality}, we obtain a pair of maps 
	\begin{align*}
		u_\infty: W'/W \cong V_\infty \lra W(D_\infty)/W \\
	u_0 : W'/(W \otimes \mathcal L_\sigma^{-1}) \cong V_0 \lra W(D_0)/W.
	\end{align*}
	Conversely, we now claim that the triple \((W, u_0, u_\infty)\) determines the Koszul data \\ \((W, W', \iota_0, \iota_\infty)\). The vector bundle \(W'\) and map \(\iota_0\) is obtained by forming the pullback
	\begin{center}
	\begin{tikzcd}
		0 \arrow[r] & W \arrow[d, equal] \arrow[r, "\iota_0"] &  W' \arrow[d] \arrow[r]  \arrow[dr, phantom, "\lrcorner", very near start] & V_\infty \arrow[d, "u_\infty"] \arrow[r] & 0 \\
		0 \arrow[r] & W \arrow[r] & W(D_\infty) \arrow[r] & W(D_\infty)/W \arrow[r] & 0.
	\end{tikzcd}
	\end{center}
	We can also form the pull back with the map \(u_0\): 
	\begin{center}
	\begin{tikzcd}
		0 \arrow[r] & W \otimes \mathcal L_\sigma^{-1} \arrow[d, equal] \arrow[r, "\iota_\infty"] &  W'' \arrow[d] \arrow[r]  \arrow[dr, phantom, "\lrcorner", very near start] & V_0 \arrow[d, "u_0"] \arrow[r] & 0 \\
		0 \arrow[r] & W \otimes \mathcal L_\sigma^{-1}\arrow[r] & (W \otimes \mathcal L_\sigma^{-1})(D_0) \arrow[r] & W(D_0)/W \arrow[r] & 0
	\end{tikzcd}
	\end{center}
	We claim that \(W' \cong W''\). To prove this claim let \(\alpha_n: E \lra E\) be the map given by addition by \(n(\sigma - b)\) under the group law of \(E\). Then by Lemma \ref{lem:divisor relation} we have that \(\alpha_n^*(W(D_\infty)/W) = W(D_0)/W\), and \(\alpha_{n+1}^*(W(D_\infty)) = W\otimes \mathcal L_\sigma^{-1}(D_0)\). Also by assumption on the framing sheaf we have \(\alpha_n^* V_\infty \cong V_0\). Now consider the following square:
	\begin{center}
	\begin{tikzcd}
		W(D_\infty) \arrow[r] \arrow[d,"\alpha^*_{n+1}"]& W(D_\infty)/W. \arrow[d, "\alpha_n^*"]\\
		(W \otimes \mathcal L_\sigma^{-1})(D_0) \arrow[r] & W(D_0)/W 
	\end{tikzcd}
	\end{center}
	This square commutes by our choice of identification of
	\[
	(W \otimes \mathcal L_\sigma^{-1}(D_0))/(W \otimes \mathcal L_\sigma^{-1}) \overset{\sim}{\lra} W(D_0)/W.
	\]
	Therefore by the universal property of pullbacks we get an isomorphism \(\alpha_n^* W' \cong W''\). Repeating the argument with \(\alpha_{-n}^*\) we also find that \(\alpha_{-n}^* W'' \cong W'\). Combining these two isomorphisms we get that \(W' \cong W''\) as desired. 

	The map \(\iota_\infty\) fits into the first pullback square and implies there exists a pair of morphisms \(\eta_\infty\) and \(\tilde{v}_\infty\) fitting into the commutative diagram
	\begin{center}
	\begin{tikzcd}
		W \otimes \mathcal L_\sigma^{-1}
		\arrow[drr, bend left, "\tilde{v}_\infty"]
		\arrow[ddr, bend right, "\eta_\infty"]
		\arrow[dr, "{\iota_\infty}" description] & & \\
		& W' \arrow[r] \arrow[d] \arrow[dr, phantom, "\lrcorner", very near start]
		& V_\infty \arrow[d, "u_\infty"] \\
		& W(D_\infty) \arrow[r]
		& W(D_\infty)/W
	\end{tikzcd}
	\end{center}
	The map \(\tilde{v}_\infty\) is determined by its restriction to \((W \otimes \mathcal L_\sigma^{-1})|_{D_\infty} \cong W|_{D_\infty}\), which we call \(v_\infty\). By the commutativity of the above diagram we get the relation required for a prolonged \(\mathcal L_\sigma\)-twisted Higgs bundle. Similarly, looking at the relation of the map \(\iota_0\) with respect to the pullback square defined with respect to the divisor \(D_0\) we get another prolonged Higgs bundle \((\eta_0, u_0, v_0)\).  

	The relation between \(\iota_0\) and \(\iota_\infty\) in Proposition \ref{prop:Koszul duality} will also imply a relation between the Higgs fields \(\eta_0\) and \(\eta_\infty\). The commutativity of the following diagram
	\begin{center}
	\begin{tikzcd}
		W \otimes \mathcal L_\sigma^{-1}
		\arrow[dr, "{\iota_\infty}"] &\\
		W \arrow[r, "\iota_0"] \arrow[from = u, dr, "\eta_\infty", swap, crossing over, bend right] \arrow[d, equal] & W' \arrow[d] \\
		W \arrow[r, hook, "t_\infty"] & W(D_\infty)
	\end{tikzcd}
	\end{center}
	implies that away from the divisor \(D_\infty\), the Higgs field \(\eta_\infty = t_\infty \circ (\iota_0^{-1} \circ \iota_\infty)\). Similarly we get that \(\eta_0 = t_0\circ (\iota_\infty^{-1} \circ \iota_0)\). As described before, the short exact sequences with sheaves \(W(D_\infty)\) and \(W(D_0)\) are related by the pullback by the shift map \(\alpha_n\) which intertwines \(t_0\) and \(t_\infty\). Hence we obtain the desired relation.
\end{proof}

\section{Geometric construction of the RS Lax matrix} \label{sec: laxmatrix} 
In this section we aim to show that for \(V\) of the form \(\OO_{p_0}^{k} \oplus \OO_{p_\infty}^k\), we can recover the spin Ruijsenaars-Schneider Lax matrix from the Higgs description of the moduli space \(\mathsf{RS}_{\sigma, n}(E,V)\). This the first step in showing that \(\mathsf{RS}_{\sigma, n}(E,V)\) is isomorphic to a completion of the spin RS phase space. 

We first note that the moduli of semi-stable vector bundles of rank \(n\) and degree \(0\) is birational to the \(n\)-fold symmetric product \(S^n E\), so on a generic open set of \(\mathsf{RS}_{\sigma, n}(E,V)\) the vector bundle \(W = p_* \mathscr F(-E_\infty)\) decomposes into a direct sum of degree \(0\) line bundles
\[W \cong \bigoplus_{i=1}^{n}\OO_E(q_i - b)\]
for distinct points \(q_i \in E\) and the fixed base point \(b \in E\). It is on this open locus that we calculate the Higgs fields \(\eta_0\) and \(\eta_\infty\) in coordinates on \(E\). These will be coordinates given by elliptic functions on the smooth cubic curve, and for the nodal and cuspidal degenerations of \(E\) the elliptic functions will degenerate to trigonometric and rational functions respectively. In order to construct the Lax matrix we first turn our attention to some useful calculations.
\subsection{Lemmas on matrices of elliptic functions}
To begin, we need some specific calculations that will be used in the construction of the Lax matrix. The first is a generalization of the formula for the determinant of a Cauchy matrix. 

An \(n \times n\) matrix \(A\) is called a Cauchy matrix if the \(i,j\)-component is of the form 
\[A_{i,j} = \frac{1}{x_i - y_j}.\]
The \(x_i, y_j\) for \(i, j = 1, \ldots n\) can either be interpreted as distinct collections of numbers, or generators for a ring of rational functions. In either case, the classical Cauchy determinant formula states that for a Cauchy matrix \(A\)
\[
\det A = \frac{\prod_{1 \leq i < j \leq n}(x_j - x_i)(y_i - y_j)}{\prod_{i, j = 1}^n (x_j - y_i)}.
\] 
The following lemma provides a generalization of this matrix identity to analogous matrices with entries given by elliptic functions. Recall our notation for \(\sigma(z)\), the Weierstrass \(\sigma\)-function with periods \(\{1, \tau\}\) defining the elliptic curve \(E\). For complex numbers \(q_i, r_i\) \(i = 1, \ldots n\) and \(q_\infty\) (equivalently points \(q_i, r_i\) and \(q_\infty\) in \(E\)) we call an \(n \times n\) matrix \(H(\lambda)\) with entries of the form
\[
H(\lambda)_{i,j} = \frac{\sigma\left(q_i - r_j - \lambda\right)}{\sigma(\lambda - q_\infty)\sigma(q_i - r_j - q_\infty)}
\]
an \emph{elliptic Cauchy matrix with spectral parameter \(\lambda\)}. Matrices of this form were already considered by Ruijsenaars in \cite{MR851627}, and an analogue of the Cauchy determinant formula holds. A straightforward generalization of the argument in loc. cit implies the following determinant identity for its minors:

\begin{lem} \label{lem:determinant}
	For an elliptic Cauchy matrix with spectral paremeter \(\lambda\), \(H(\lambda)\), the determinant is given by the formula
	\begin{align} \label{elliptic cauchy matrix}
		\det H(\lambda) = \frac{\sigma\left(\lambda + \displaystyle \sum_{i}(q_i-r_i)\right)}{\sigma(\lambda)}\frac{\displaystyle \prod_{i < j} \sigma(q_i - q_j)\sigma(r_j-r_i)}{\displaystyle \prod_{i,j}\sigma(q_i-r_j)}. 
	\end{align}
	More generally, the determinant of the \((k,l)\) minor obtained by deleting row \(k\) and column \(l\) is given by
	\begin{align} \label{eqn: minors of cauchy}
		\det \mathrm{Min}_{(k,l)} \, H(\lambda) = \frac{\sigma\left(\lambda + \displaystyle q_k - r_l + \sum_{i \neq k, l}(q_i-r_i)\right)}{\sigma(\lambda)}\frac{\displaystyle \prod_{\substack{i < j \\ i \neq k, \, j \neq l}} \sigma(q_i - q_j)\sigma(r_j-r_i)}{\displaystyle \prod_{\substack{i \neq k, \, j \neq l}}\sigma(q_i-r_j)}. 
	\end{align}
\end{lem}
A version of this determinant identity was known by Frobenius and it follows by induction using Fay's trisecant identity. 

The last result we need before constructing the Lax matrix is a version of a calculation that was used in \cite{Hasegawa1997} in calculating inverses of matrices of intertwining vectors. 

\begin{lem} \label{lem:matrix product}
	Let \(F(z) = (f_{i,j}(z))_{i,j}\) be an invertible \(n \times n\) matrix, with inverse entries denoted \(F^{-1}(z) = (\bar{f}_{i,j}(z))_{i,j}\). Then
	\begin{align} \label{eqn:matrix product}
		(F^{-1}(z) \cdot F(z + u))_{i, j} &= \sum_{k=1}^n \bar{f}_{i,k}(z)f_{k, j}(z + u) \\ 
		&= 
		\frac{\det \bordermatrix{
		 & &(i-1)      & i               & (i+1)         &        \cr
		 &\ldots &f_{1,i-1}(z) & f_{1, j}(z + u) & f_{1, i+1}(z) & \ldots \cr
		 &      &   \vdots   &   \vdots        & \vdots        &        \cr
		 &\ldots & f_{n,i-1}(z) & f_{n, j}(z + u) & f_{n, i+1}(z) & \ldots \cr
		}}{\det F(z)}.
	\end{align}
\end{lem}

\begin{proof}
	Let \(\widetilde{F}_{j,i}(z)\) be the matrix in the numerator of the right hand side of the equation obtained by swapping the \(i\)th column of \(F(z)\) with the \(j\)th column of \(F(z+u)\). Evaluate the determinant of \(\widetilde{F}_{j,i}(z)\) by cofactor expansion along the \(i\)th column to obtain
	\[\sum_{k=1}^n f_{k, i}(z + u) \Cof_{k,j}(\widetilde{F}_{i,j}(z))\]
	where we denote the \((i,j)\) cofactor of \(A\) by \(\Cof_{i,j}(A)\). But \(\Cof_{k,j}(\widetilde{F}_{i,j}(z)) = \Cof_{k,j}(F(z))\), so cofactor terms multiplying \(f_{k,i}(z+u)\) in the above sum are simply matrix elements of the inverse of \(F(z)\) by Cramer's rule.
\end{proof}

\subsection{Construction of the Lax matrix}
We now have all the technical tools in place to prove the main theorem of this section. The construction and proof are inspired by the calculations in \cite{math/9909079}. Their parametrization of the Skylanin-Hasegawa \(\mathbf L\) operator is given in terms of intertwining vectors, but it can be seen to be equivalent to this paper's parametrization of in terms of sections of line bundles by Weierstrass elliptic functions rather than \(\theta\) functions with characteristics. The author of loc. cit also uses the language of elementary modifications. Note that the two short exact sequences of Proposition \ref{prop:Koszul duality} exhibit \(W'\) as an elementary modification of \(W\) at the divisors \(D_0\) and \(D_\infty\) which connects our approach with previously done analysis. 

Instead of considering the Lax matrix in terms of the maps \(\iota_0\) and \(\iota_\infty\) we instead choose to analyze the composition
\begin{align}\label{eqn:composition}
	W \overset{\eta_\infty}{\lra} (W \otimes \mathcal L_\sigma)(D_0) \overset{\eta_0 \otimes id}{\lra} W(D_0 + D_\infty).
\end{align}
Note that by Theorem \ref{thm:higgs description} \(\eta_0\) is a translation of \(\eta_\infty^{-1}\) which suggests the usefulness of Lemma \ref{lem:matrix product}. We first concentrate on the spinless Ruijsenaars-Schneider system by considering framed RS sheaves with simple framing given by \(V = \OO_{p_0}\oplus \OO_{p_\infty}\) for two points \(q_0\), \(q_\infty\) related by equation \eqref{eqn:divisor relation}.

\begin{prop} \label{prop:geometric lax matrix}	
	Let \(\eta_0\) and \(\eta_\infty\) be the associated Higgs fields for an RS spectral sheaf in \(\mathsf{RS}_{\sigma, n}(E, \OO_{p_0} \oplus \OO_{p_\infty})\) where \(E\) is a smooth Weierstrass cubic. On the locus of decomposable vector bundles, the composition \((\eta_0 \otimes id) \circ \eta_\infty\) of equation \eqref{eqn:composition} equals the RS Lax matrix
	\[
	((\eta_0 \otimes id) \circ \eta_\infty)_{k,k'} = \frac{\sigma(z + \hbar + q_k - q_{k'})}{\sigma(z)} \prod_{l \neq k} \frac{\sigma(\hbar + q_l - q_{k'})}{\sigma(q_{l} - q_k)} \exp P_k
	\]
	with \(\hbar = (\sigma - b)\). 
\end{prop}

\begin{proof}
	We first consider the Higgs field \(\eta_\infty\) in the open locus of decomposable vector bundles. As discussed above, \(W\) will decompose into the direct sum
	\[
	W \cong \bigoplus_{j = 1}^n \OO_E(q_j - b).
	\]
	Hence the Higgs field \(\eta_\infty\) defines a map
	\[
	\eta_\infty: \bigoplus_{j=1}^n \OO_E(q_j - b) \lra \bigoplus_{k = 1}^n \OO_E(q_j - b) \otimes \OO_E(q_\infty + (n+1)(\sigma-b)).
	\]
	For a point \(\sigma \in E\) and its associated translation automorphism \(E \overset{ \bullet + \sigma}{\lra} E\) we will denote the image of a point \(q\) under the translation \(n\sigma\) by \(q^{n\sigma}\). The \((k,j)\) entry of the Higgs field is a map between line bundles
	\[
	\OO_E(q_j - b) \lra \OO_E\left((q_k -b)+q_\infty^{(n+1)\sigma}\right)
	\]
	or equivalently a global section of the line bundle \\ \(H^0\left(E, \OO_E\left((q_k - q_j) + q_\infty^{(n+1)\sigma}\right)\right)\). By the discussion in Section \ref{subsec: line bundles}, we can identify such a section with an elliptic function satisfying the appropriate vanishing and quasiperiodic condition. The zeroes should be equal to lattice multiples of the point \((q_k - q_j) + q_\infty^{(n+1)\sigma}\), and being degree \(1\) forces the elliptic function to transform linearly as a theta function with characteristic. 
	
	In particular, we can choose a basis of sections in terms of such theta functions \(\thetachar{a}{b}(z)\) or by equation \eqref{eqn:sigma shift}, in terms of translations of the Weierstrass \(\sigma\)-functions. We can therefore simply compare the zeroes of sections of \(\OO_E\left((q_k - q_j) + q_\infty^{(n+1)}\right)\) with those of \(\sigma(z)\) and conclude that the \((k,j)\) component of \(\eta_\infty\) is given by
	\[
	(\eta_\infty)_{k,j} = C_j\vartheta_{k,j} \sigma\left(z - (q_k - q_j) - q_\infty^{(n+1)\sigma}\right).
	\]
	where the constant \(C_j\) is the constant appearing in the trivial theta function \(\vartheta_{k,j}\). We can write \(C_j = \exp P_j\) to explicitly yield the momentum dependence.
	By the relation between between the Higgs fields of Theorem \ref{thm:higgs description} we get that :
	\[
	(\eta_{0})_{k,j} = C_j\vartheta_{k,j} \sigma\left(z - (q_k -q_j) - q_0^\sigma \right)= \vartheta_{k,j} \sigma\left(z - (q_k -q_j) - \left(q_\infty^{(n+1)\sigma} + n \hbar\right)\right)
	\]
	where \(n\hbar = n(\sigma - b)\) is the difference of \(q_0\) and \(q_\infty\). Therefore since \(\eta_0^{-1}\) and \(\eta_\infty\) are generically given by shifts of the same matrix \(F(z)\) and the composition 
	\[
	((\eta_0 \otimes id) \circ \eta_\infty)_{k,k'} = (F^{-1}(z) \cdot F(z + u))_{k, k'} = \sum_{l=1}^n \bar{f}_{k,l}(z)f_{l, k'}(z + u)
	\]
	can be calculated up to a factor of a trivial theta function using the following result from \cite{Hasegawa1997}:

	\begin{lem}[Equation 31 \cite{Hasegawa1997}] \label{lem:intertwining lemma}
	Let \(\Xi(z) = (\phi_{\lambda, j}^{\lambda + \hbar \bar{\epsilon}_i})_{i,j}^n \) be the matrix of intertwining vectors with appropriately chosen \(\lambda \in \mathbb C^n\) so that \(q_i = \langle \lambda, \bar{\epsilon_i}\rangle\). Then
	\begin{align*}
		\left(\Xi(z)^{-1} \cdot \Xi(z + n \hbar)\right)_{i,j}= \frac{\sigma(\hbar +  z - q_\infty + q_j - q_i)}{\sigma(z - q_\infty)} \prod_{\substack{l \neq j}} \frac{\sigma(\hbar + q_l - q_j)}{\sigma(q_l - q_i)} \exp{P_i}.
	\end{align*}
	\end{lem}

	The theorem follows by first equating the intertwining vectors up to trivial theta function with shifts of Weierstrass \(\sigma\)-functions using \eqref{eqn:sigma shift}, and identifying \(u = n \hbar\), and shifting \(z \mapsto z - q_\infty\) in the formula for the Lax matrix. 	
\end{proof}

By considering the degenerations of \(E\) into the cuspidal and nodal Weierstrass cubic curve, we obtain the associated degenerations of the Weierstrass \(\sigma\)-function into rational and trigonometric functions in Section \ref{subsec: line bundles}. These degenerations together with the degenerations of elliptic function identities to trigonometric and rational identities in appendix B of \cite{ruijsenaars1987} yields the following corollary.

\begin{cor}\label{cor:singular lax}
	For singular Weierstrass cubic curves, the composition \((\eta_0 \otimes id) \circ \eta_\infty\) recovers the rational and trigonometric spinless RS Lax matrix in the case of cuspidal and nodal \(E\) respectively.
\end{cor}

\begin{rmk}\label{rmk: alt Lax matrix}
	Note that the resulting Lax matrix in the above spectral description of the RS system comes naturally in the factorized form of \cite{Hasegawa1997}:
	\[
	(\widetilde{L}_{RS})_{i,j} = \theta_1(0)/\theta_1(\hbar)\sum_k g_{i,k}^{-1}(z,q) g_{k,j}(z + N \hbar, q) e^{p_j/c}
	\] spelled out by the authors of \cite{Vasilyev:2018cyt}. The interpretation of the matrix \(g_{i,k}(z, q)\) in loc. cit as a Higgs field suggests that there should be an integrable system describing the isospectral flows of the \(g_{i,j}\) which can reconstruct the RS integrable system when forming the above combination to the obtain the RS Lax matrix.

	The simple identity that
	\[
	\OO_E\left((q_k - q_j) + q_\infty^{(n+1)\sigma}\right) \cong \OO_E(-(q_j - q_k)) \otimes \OO_E\left(q_\infty^{(n+1)\sigma}\right)
	\]
	allows us to identify the sections of the above line bundle with the product
	\[
	\frac{\sigma(q_j - q_k + \lambda)}{\sigma\left(\lambda - q_\infty^{(n+1)\sigma}\right) \sigma(q_j - q_k)}
	\]
	Therefore in this basis, by Cramer's rule and Lemma \ref{lem:determinant} we find that the inverse of the Higgs field \((\eta_0)^{-1}\), or simply \(\eta_\infty(z + n \hbar)\) can be identified with the original Ruijsenaars Lax matrix in \cite{ruijsenaars1987} up to a factor of 
	\[
	\widetilde{P}_i = \exp(\theta_i) \prod_{k \neq i}\frac{\sigma(q_i - q_k - \hbar)}{\sigma(q_i - q_k)}.
	\]
	It is possible to redefine \(\theta_i\) by 
	\[
	\theta_i \mapsto \theta_i + \log \prod_{k \neq i}\frac{\sigma(q_i - q_k - \hbar)}{\sigma(q_i - q_k)}
	\]
	and absorb the product into arbitrary constants scaling the entries of the lax matrix \(\eta_\infty(z+n \hbar)\). In this way we get both the factorized, and non-factorized perspectives on the Lax matrix from the above construction.
\end{rmk}

A simple corollary of the above result is that we can identify the spin RS system Lax matrix  Specifically, for the framing \(\OO_{q_0}^r \oplus \OO_{q_\infty}^r\) we recover the spin Ruijsenaars-Schneider system with the same argument. 

\begin{cor} \label{cor:spin Lax matrix}
The composition of equation \eqref{eqn:composition} with framing \(V = \OO_{q_0}^n \oplus \OO_{q_\infty}^n\) recovers the spin RS Lax matrix
	\[
	\left((\eta_0 \otimes id) \circ \eta_\infty\right)_{k,k'} = f_{k,k'} \frac{\sigma(z + u/n + q_k - q_{k'})}{\sigma(z)} \prod_{l \neq k} \frac{\sigma(u/n + q_l - q_{k'})}{\sigma(q_l - q_k)}
	\]
	where \(f_{k, k'}\) are given by \(f_{k,k'} = (u_0(v_0(q_0)))_{k,k'} \cdot (u_\infty(v_\infty(q_\infty)))_{k',k}\)
\end{cor}
\begin{proof}
	When considering RS spectral sheaves with framing \(V = \OO_{q_0}^n \oplus \OO_{q_\infty}^n\) we have the additional data of \(u_0, v_0, u_\infty, v_\infty\) of Theorem \ref{thm:higgs description} that scales the principal parts of the matrix representations of \(\eta_0\) and \(\eta_\infty\). Isolating the parts with singularities along \(q_0\) and \(q_\infty\) we find that the entries in the unscaled Higgs field \(\eta_0\) gain a factor coming from the \(n \times n\) matrix
	\[
	\mathbb C^n \cong W|_{q_0} \overset{v_0}{\lra} V_0 \overset{u_0}{\lra} W|_{q_0} \cong \mathbb C^n
	\]
	and similarly for \(\eta_\infty\).
\end{proof}
\begin{rmk}
	In the rational limit, our expression for the spin RS Lax matrix matches with that obtained in \cite{Pampolina}. 
\end{rmk}
\section{Flows on the space of RS spectral sheaves} \label{sec:flows}




In this section we show that we can realize the Ruijsenaars-Schneider Hamiltonian flows on the moduli space \(\mathsf{RS}_{\sigma, n}(E, V)\) by considering so called tweaking flows on spectral sheaves. These kinds of flows on spectral data of integrable systems were used in \cite{PMIHES_2001__94__87_0} and \cite{Donagi1996} to realize flows of generalized Drinfeld-Sokolov and the KP hierarchies uniformly on the spectral sheaves of the systems. More recently they were used in \cite{MR2377220} to describe the Calogero Moser flows on the space of CM spectral sheaves. Our aim is to mimic this description for Ruijsenaars-Schneider spectral sheaves.

\subsection{Tangent Space to the RS phase space}

Our approach to describing the tangent space to \(\mathsf{RS}_{\sigma, n}(E, V)\) follows the techniques of the description of infinitesimal deformations of Higgs bundles in \cite{doi:10.1112/plms/s3-62.2.275}. 

The space of infinitesimal deformations of the twisted Higgs pair \((\mathcal E, \phi: \mathcal E \lra \mathcal E \otimes \mathcal L)\) on a scheme \(X\), denoted \(D_{(\mathcal E, \phi)}\), fits into the exact sequence
\[
0 \lra H^0(X, \End (\mathcal E)\otimes \mathcal L) \lra D_{(\mathcal E, \phi)} \overset{g}{\lra} H^1(X, \End \mathcal E).
\]
The map \(g\) takes a deformation of the pair \((\mathcal E, \phi)\) and remembers only the deformation of the vector bundle \(\mathcal E\). The kernel of \(g\) is precisely space of deformations of the pair \((\mathcal E, \phi)\) which fix the vector bundle \(\mathcal E\). These correspond to deformations of the Higgs field \(\phi\) which lie in the affine space \(H^0(X, \End (\mathcal E)\otimes \mathcal L)\). 

In order to describe the tangent space to the moduli of Higgs bundles at the pair \((\mathcal E, \phi)\) one needs to further mod out by the infinitesimal action of \(\Aut \mathcal E\) on \(D_{(\mathcal E, \phi)}\). The result obtained in \cite{doi:10.1112/plms/s3-62.2.275} is that the tangent space lies in the 5-term exact sequence 
\begin{align*}
H^0(X, \End \mathcal E) \rightarrow H^0(X, \End (\mathcal E) \otimes \mathcal L) \rightarrow T_{(\mathcal E, \phi)}\mathcal M_{Higgs}(X) \rightarrow \\
H^1(X, \End \mathcal E) \rightarrow H^1(X, \End(\mathcal E) \otimes \mathcal L)
\end{align*}
With the first and last maps induced by commutation with \(\phi\). Roughly, the infinitesmal action of \(\Aut \mathcal E\) acting on the pair \((\mathcal E, \phi)\) preserving the Higgs field will commute with \(\phi\). Similarly, the deformations of the vector bundle \(\mathcal E\) which are deformations of the pair \((\mathcal E, \phi)\) are elements of \(H^1(X, \End \mathcal E)\) which lie in the kernel of the map to \(H^1(X, \End(\mathcal E) \otimes \mathcal L)\). Our aim now is to provide a similar description of the tangent space to the RS phase space in the Hitchin coordinates.

Let \(\widetilde{\mathsf{RS}}_{\sigma, n}(E, V)\) be the auxiliary space of pairs \((\mathscr F, \beta)\) where
\[
\beta: \mathscr F|_{E_0 \cup E_\infty} \overset{\sim}{\lra} V
\]
is a fixed framing isomorphism. The map \(\widetilde{\mathsf{RS}}_{\sigma, n}(E, V) \lra \mathsf{RS}_{\sigma, n}(E, V)\) exhibits \\ \(\widetilde{\mathsf{RS}}_{\sigma, n}(E, V)\) as an \(\Aut(V)\) torsor over \(\mathsf{RS}_{\sigma, n}(E, V)\). Elementary deformation theory implies that the tangent space to the moduli of sheaves on \(S_\sigma\) at \(\mathscr F\) is given by the vector space \(\Ext_{\OO_{S_\sigma}}^1(\mathscr F, \mathscr F)\). The short exact sequence
\[
0 \lra \mathscr F(-E_0 - E_\infty) \lra \mathscr F \lra \mathscr F|_{E_0 \cup E_\infty} \lra 0
\]
induces a long exact sequence
\begin{align*}
	\Hom(\mathscr F, \mathscr F|_{E_0 \cup E_\infty}) \rightarrow \Ext^1(\mathscr F, \mathscr F(-E_0 - E_\infty)) \overset{\cup \theta}{\rightarrow} \\
	\Ext^1(\mathscr F, \mathscr F) \overset{f}{\rightarrow} \Ext^1(\mathscr F, \mathscr F|_{E_0 \cup E_\infty}).
\end{align*}

The rightmost map can be interpreted as restricting the deformation of \(\mathscr F\) to the sections \(E_0\) and \(E_\infty\) of \(S_\sigma\). Hence the deformations of \(\mathscr F\) that do not change the framing with \(V\) are given by the kernel of \(f\), which is precisely the image of \(\Ext^1(\mathscr F, \mathscr F(-E_0 - E_\infty))\) in the vector space \(\Ext^1(\mathscr F, \mathscr F)\). This gives a description of the tangent space \(T_{\mathscr F}\mathsf{RS}_{\sigma, n}(E, V)\) in the standard spectral sheaf coordinates on the RS phase space.

In this description, the symplectic structure on the phase space is induced by a section \(\theta\) of \(K^\vee_{S_\sigma} = \OO_{S_\sigma}(E_0 + E_\infty)\) giving a Poisson structure on the surface \(S_\sigma\). Following \cite{MR1346215}, the map in the above exact sequence 
\[
\Ext^1(\mathscr F, \mathscr F(-E_0 - E_\infty)) \overset{\cup \theta}{\lra} \Ext^1(\mathscr F, \mathscr F)
\]
defined by \(\theta\) defines the anchor map of the Poisson structure, and it is an isomorphism onto its image---hence defining a symplectic structure. The pairing defined by this Poisson structure coincides with the Poisson bivector pairing on \(T_{\mathscr F}^*\mathsf{RS}_{\sigma, n}(E,V)\) defined by the composition
\begin{eqnarray*}
	\Ext^1(\mathscr F, \mathscr F(-E_0 - E_\infty)) \otimes \Ext^1(\mathscr F, \mathscr F(-E_0 - E_\infty)) \rightarrow \Ext^2(\mathscr F , \mathscr F(-2 E_0 - 2 E_\infty)) \\
	\overset{\theta}{\rightarrow} \Ext^2(\mathscr F, \mathscr F(-E_0 - E_\infty)) \overset{\Tr}{\rightarrow} H^2(S_\sigma, \OO_{S_\sigma}(-E_0 - E_\infty)) \overset{\Res}{\rightarrow} \mathbb C
\end{eqnarray*}
which is non-degenerate on the dual to the tangent space. 
\begin{rmk}
	In fact the above calculations show that \(\mathsf{RS}_{\sigma,n}(E, V)\) forms a locally closed subscheme in a symplectic leaf of the moduli space of pure 1-dimensional sheaves on \(S_\sigma\) with fixed Hilbert polynomial, where the Poisson structure is induced by the Poisson structure on \(S_\sigma\). 
\end{rmk}

\begin{rmk}
	The standard symplectic structure on the RS phase space, given by the 2-form \(\omega = \sum_i (d \theta_i / \theta_ i) \wedge d q_i\) coincides with the the symplectic structure defined above, at least on a suitable open locus. We can identify an open locus of the phase space of the n-particle RS system with \((S_\sigma^*)^{[n]}\), the Hilbert scheme of n points on \(S_\sigma\). The \(n\) points describe the \(n\) separate positions on \(E\) and momenta living in the open \(\mathbb C^*\)-subbundle \(S_\sigma^* \subset S_\sigma\). In coordinates, the Poisson structure is given by the above formula for \(\theta\). 

	Again following the arguments in Section 4 of \cite{MR1660136}, we may restrict to the open locus \(\mathcal J\) of \(\mathsf{RS}_{\sigma, n}(E, V)\) consisting of RS spectral sheaves of the form \(i_* \mathcal L\) for a line bundle \(\mathcal L\) on an RS spectral curve \(i: \Sigma \hookrightarrow S_\sigma\). A collection of \(n\) points on \(S_\sigma\) defines a unique curve \(\Sigma\) passing through those points, and a line-bundle \(\mathcal L\) which differs from \(\OO_\Sigma\) away from those points. In  \cite{MR1660136} the author proves that this map \((S_\sigma^*)^{[n]} \rightarrow \mathcal J\) is a symplectomorphism on an open locus. It is in this way that we see the compatibility of the above symplectic structure with the classically known symplectic structure on the RS phase space. 
\end{rmk}

We will define the RS tweaking flows in the coordinates of the spectral sheaves, but in order to show the flows are Hamiltonian we need to also understand the tangent space in the coordinates given by Theorem \ref{thm:higgs description}.

We obtain the tangent space description in the Hitchin description by following the outline of the Nitsure calculation. By Theorem \ref{thm:higgs description} an RS spectral sheaf \(\mathscr F\) determines the Higgs data \((W, \eta_0)\) (noting that \(\eta_\infty\) is determined by \(\eta_0\)). Just as before, denote \(D_{(W, \eta_0)}\) the space of deformations of the pair \((W, \eta_0)\), and notice that we can fit the set of deformations into the exact sequence
\[
0 \lra H^0(X, \End(W) \otimes \mathcal L_\sigma^{-1}(D_\infty)) \lra D_{(W, \eta_0)} \overset{f}{\lra} H^1(X, \End(W)(-D_\infty))
\]
where the map \(f\) remembers only the deformations \(\widetilde{W}\) of \(W\) that is required to fit in the exact sequence
\[
0 \lra \widetilde{W} \lra \widetilde{W}' \lra V_\infty \lra 0.
\]

Taking into account the action of \(\Aut_{D_\infty + D_0} W\) of automorphisms which equal the identity on the divisor \(D_\infty + D_0\), a simple cocycle calculation mimicing that of \cite{doi:10.1112/plms/s3-62.2.275} gives us that the tangent space to \(\mathsf{RS}_{\sigma, n}(E, V)\) fits into a 5-term exact sequence as before:

\begin{prop}\label{prop:higgs tangent space}
	Let \(T_{(W, \eta_0, \eta_\infty)}\) denote the tangent space of \(\mathsf{RS}_{\sigma, n}(E, V)\) at the point \((W, \eta_0, \eta_\infty)\). Then \(T_{(W, \eta_0, \eta_\infty)}\) fits into either of the exact sequences
	\begin{align*}
		H^0(E, \End(W)(-D_\infty - D_0)) \overset{[\eta_0, \bullet]}{\rightarrow} H^0(E, \End(W) \otimes \mathcal L_\sigma^{-1}(D_\infty)) \rightarrow T_{(W, \eta_0, \eta_\infty)} \\
		\rightarrow H^1(E, \End(W)(-D_\infty - D_0)) \overset{[\eta_0, \bullet]}{\rightarrow} H^1(E, \End(W) \otimes \mathcal L_\sigma^{-1}(D_\infty))
	\end{align*}
	or 
	\begin{align*}
		H^0(E, \End(W)(-D_\infty - D_0)) \overset{[\eta_\infty, \bullet]}{\rightarrow} H^0(E, \End(W) \otimes \mathcal L_\sigma(D_0)) \rightarrow T_{(W, \eta_0, \eta_\infty)} \\
		\rightarrow H^1(E, \End(W)(-D_\infty - D_0)) \overset{[\eta_\infty, \bullet]}{\rightarrow} H^1(E, \End(W) \otimes \mathcal L_\sigma(D_0)).
	\end{align*}
	Furthermore, the symplectic form at the point \((W, \eta_0, \eta_\infty)\) is given by 
	\begin{align} \label{eqn:symplectic form}
	\omega_{(W, \eta_0, \eta_\infty)}((s_0, \theta_0), (s_\infty, \theta_\infty)) = \Res \Tr(s_0 \circ \eta_0 \circ \theta_\infty - s_\infty \circ \eta_\infty \circ \theta_0)
	\end{align}
	With the pairs \((s_0, \theta_0)\) and \((s_\infty, \theta_\infty)\) representing tangent vectors in the first and second short exact sequences respectively. 
\end{prop}

\begin{proof}
	The proof that the tangent space fits into the 5-term exact sequence can be gathered from a detailed calculation via cocycles. The infinitesimal action of \(\Aut_{D_\infty + D_0} W\) fixing the data \((W, \eta_0, \eta_\infty)\) commutes with the Higgs fields, and hence the map \[H^0(E, \End(W)(-D_\infty - D_0)) \lra D_{(W, \eta_0, \eta_\infty)}\] factors through \(H^0(X, \End(W) \otimes \mathcal L_\sigma^{-1}(D_\infty))\) (and the corresponding space of global sections for the Higgs field \(\eta_\infty\)). Also the deformations of \(W\) that fix the framing along \(D_\infty + D_\infty\) of the Higgs field must correspond to classes in \(H^1(E, \End(W)(-D_\infty - D_0))\) that commute with \(\eta_0\) (and \(\eta_\infty\)).

	An alternative approach comes from the general result that the tangent space to the moduli space of \(\mathcal L_\sigma^{-1}(D_\infty)\)-twisted Higgs bundles with fixed framing along \(D_\infty + D_0\) corresponds to the first hypercohomology of the 2-term complex:
	\[
	C^\bullet: \,\,\, \End(W)(-D_\infty - D_0) \lra \End(W) \otimes \mathcal L_\sigma^{-1}(D_\infty).
	\]
	To compare this result with the statement of the proposition take the cohomology long exact sequence of the short exact sequence of complexes
	\begin{center}
	\begin{tikzcd}
	0 \arrow[d] &  0 \arrow[d] \arrow[r] &  0 \arrow[d]\\
	A^\bullet: \arrow[d] & 0 \arrow[r] \arrow[d] & \End(W) \otimes \mathcal L_\sigma^{-1}(D_\infty) \arrow[d] \\
	C^\bullet: \arrow[d]& \End(W)(-D_\infty - D_0) \arrow[r] \arrow[d] & \End(W) \otimes \mathcal L_\sigma^{-1}(D_\infty) \arrow[d] \\
	B^\bullet: \arrow[d] & \End(W)(-D_\infty - D_0) \arrow[r] \arrow[d] & 0 \arrow[d] \\
	0 & 0 \arrow[r] & 0.
	\end{tikzcd}
	\end{center}

	For a fixed Higgs field \(\eta_0 \in H^0(E, \End(W) \otimes \mathcal L^{-1}_\sigma(D_\infty))\) the pairing 
	\begin{align*} 
		H^1(E, \End(W)(-D_\infty - D_0)) \otimes H^0(E, \End(W) \otimes \mathcal L_\sigma^{-1}(D_\infty)) \lra \mathbb C
	\end{align*}
	given by \((s_0, \theta_\infty) = \Res \Tr(s_0 \circ \eta_0 \circ \theta_\infty)\)is non-degenerate because \(\eta_0\) is generically full rank on \(E\). Similarly we have that the pairing \((s_\infty, \theta_0) = \Res \Tr(s_\infty \circ \eta_\infty \circ \theta_0)\) is non-degenerate.

	Furthermore the trace-residue pairing vanishes on deformations of the Higgs field of the form \([\phi, \eta_0]\) or \([\phi, \eta_\infty]\) for a \(\phi \in H^1(E, \End(W)(-D_\infty - D_0))\) so the antisymmetric form given in equation \eqref{eqn:symplectic form} is non-degenerate and hence defines a symplectic form on \(\mathsf{RS}_{\sigma, n}(E, V)\).
\end{proof}
\subsection{Tweaking flows on the RS phase space} We now turn to the description of a set of natural flows on the RS phase space that will yield the RS Hamiltonian flows by Proposition \ref{prop: tweaking theorem}. 

First in a broader setting, on moduli spaces of sheaves \(\mathcal M(X)\) on a scheme \(X\) there exists an action of the Picard group of \(X\) by tensoring \(\mathscr F \mapsto \mathscr F \otimes \mathcal L\). The induced infinitesimal action yields a map from the tangent space to the identity of \(\Pic X\) to the tangent space to \(\mathscr F\)
\[
T_{\OO_X} \Pic X \cong H^1(X, \OO_X) \lra T_{\mathscr F} \mathcal M(X) \cong \Ext^1(\mathscr F, \mathscr F)
\]
which can be realized by tensoring the short exact sequence 
\[
0 \lra \OO_X \lra \mathcal D \lra \OO_X \lra 0
\]
representing a tangent vector in \(H^1(X, \OO_X)\) by the sheaf \(\mathscr F\). 

One way of generating elements of \(H^1(X, \OO_X)\) is by considering the image of the local cohomology group \(H^1_D(X, \OO_X)\) for a divisor \(D\) of \(X\). In this case we can concretely realize the deformation of \(\mathscr F\) by taking a representative of \(H^1_D(X, \OO_X) \cong \Hom(\OO_X, \widehat{\OO}_X(\infty D)/\widehat{\OO}_X)\) and tensoring with \(\mathscr F\) to obtain a map
\[
\mathscr F \lra \widehat{\mathscr F}(\infty D)/\widehat{\mathscr F}
\]
and pulling back along the short exact sequence
\[
0 \lra \mathscr F \lra \mathscr F(\infty D) \lra \widehat{\mathscr F}(\infty D)/\widehat{\mathscr F} \lra 0
\]
to obtain an element of \(\Ext^1(\mathscr F, \mathscr F)\). Let \(\mathcal E\) be algebra of Laurent series along the divisor \(D\). Specifically
\[
\mathcal E = \lim_{\lra} \widehat{\OO}_{X, D}(k D).
\]
In this way we can write \(\widehat{\mathscr F}(\infty D) = \mathscr F_{\mathcal E} = \mathscr F \otimes \mathcal E\). 
If the sheaf \(\mathscr F\) is a vector bundle, there is an informal way of interpreting the flows obtained in this way that offer a useful intuition on the construction. We can obtain \(\mathscr F\) by considering its restrictions to \(X \setminus D\) and the completed formal scheme \(\widehat{D}\), and gluing along the the formal punctured disk \(\Spec \mathcal E\) as spelled out in \cite{beauville1994}. Define an element \(\xi \in \Hom(\mathscr F, \widehat{\mathscr{F}}(\infty D)/\widehat{\mathscr F})\) as coming from an element \(\xi \in \End_\mathcal E(\mathscr F_{\mathcal E})\) by twisting the natural inclusion \(\mathscr F \hookrightarrow \widehat{\mathscr{F}}(\infty D)\) by \(\xi\). Then the deformation of \(\mathscr F\) defined by \(\xi\) can be informally interpreted as twisting the gluing between \(\mathscr F|_{X \setminus D}\) and \(\mathscr F|_{\widehat{D}}\) along \(\Spec \mathcal E\) by the element \(1 + \varepsilon \cdot \xi\) over the dual numbers \(\mathbb C[\varepsilon]/(\varepsilon^2)\). 

As described above, these flows are used to describe large classes of integrable systems in natural algebro-geometric language. In \cite{MR2377220} the authors introduce a modification of the construction of tweaking flows to deal with sheaves framed along a divisor \(D \subset X\). In this case, line bundles with trivialization along \(D\) act on framed sheaves, hence we get an infinitesimal action by the vector space \(H^1(X, \OO_X(-D))\).  Following loc. cit in order to define flows for RS spectral sheaves, we should therefore consider the tweaking flows along the divisor \(E_0 + E_\infty\) of \(S_\sigma\). In particular, functions in \(\mathcal E\) which represent germs of meromorphic functions along \(E_0 + E_\infty\) will act in a way that does not change the framing of RS spectral sheaves \(\mathscr F\) along \(E_0\) and \(E_\infty\). 

\begin{defn}\label{def:RS spectral flows}
	The sheaf \(\underline{\End}_{\mathcal E}\) on \(\mathsf{RS}_{\sigma,n}(E, V)\) with fiber \(\End_{\mathcal E}(\mathscr F_{\mathcal E})\) over \(\mathscr F\) has an anchor map to the tangent sheaf \(\underline{\Ext}^1\)
	\[
	\rho: \underline{\End}_{\mathcal E} \lra \underline{\Ext}^1
	\]
	to the tangent sheaf given by the above formula
	\begin{align*}
		&\End_{\mathcal E}(\mathscr F) \lra \Hom(\mathscr F, \mathscr F_{\mathcal E}/\mathscr F) \lra \Ext^1(\mathscr F, \mathscr F) \\
		& \xi \mapsto \{s \mapsto \xi(s_{\mathcal E}) \mod \mathscr F\}.
	\end{align*}
	This exhibits \(\underline{\End}_{\mathcal E}\) as the Lie algebroid of \emph{RS tweakings}. And we call the image of \(\mathcal E \hookrightarrow \underline{\End}_{\mathcal E}\) the algebroid of \emph{central} RS tweakings.
\end{defn}

In a general setting where the moduli space of sheaves on \(X\) describes sheaves that are supported on spectral schemes \(\widetilde{\Sigma} \subset X\), we can refine the above Picard group action to an action by the Picard groups of the spectral schemes. For example, by the discussion in Section \ref{sec:RS phase space} we can view the moduli space of Higgs bundles \(\mathcal M_{Higgs}(\Sigma)\) as representing the moduli problem of Hitchin spectral sheaves of the form \(i_*(\mathcal L)\) where \(i: \widetilde{\Sigma} \hookrightarrow T^*\Sigma\) is an inclusion of a Hitchin spectral curve, and \(\mathcal L \in \Pic \widetilde{\Sigma}\). In the Hitchin setting, or in the more general setup of sheaves supported on spectral schemes, we have a map induced by the infinitesimal action
\[
	H^1(\widetilde{\Sigma}, \OO_{\widetilde{\Sigma}}) \lra T_{\mathscr F}\mathcal M(X)
\]
where \(\widetilde{\Sigma} = \supp \mathscr F\). Or in the case of spectral sheaves framed along a divisor in \(X\), we get an infinitesimal action 
\[
	H^1(\widetilde{\Sigma}, \OO_{\widetilde{\Sigma}}(-D)) \lra T_{\mathscr F}\mathcal M_{fr}(X)
\]
where \(D\) is the divisor on \(\Sigma\) obtained by intersection with the framing divisor in \(X\). 

The moduli space \(\mathsf{RS}_{\sigma, n}(E, V)\) fits this criterion. We now turn to an explicit description of a particular class of deformations of framed RS sheaves. For a choice of framing \(V = V_0 \oplus V_\infty\), let \(D_0 + D_\infty = \sum_p [p]\) be the support divisor on \(E_0, E_\infty \cong E\). Consider \(\bigoplus_p \widehat{F}\) the algebra of germs of meromorphic functions along the fibers above the divisor \(D_0 + D_\infty\). Any germ \(\xi_p \in \widehat{F}_p\) restricts to an RS spectral curve \(\widetilde \Sigma = \supp \mathscr F\) and define a framed deformation of \(\mathscr F\). Furthermore, these deformations land in the algebroid of central RS tweakings, so in particular define a commutative hierarchy.

\begin{defn}\label{def: RS hierarchy}
	For each \(p \in (D_0 + D_\infty)\), let \(x_p\) be a local coordinate near the fiber \(F_p \subset S_\sigma\), and consider the polynomial subalgebra \(\mathbb C[x_p] \subset \OO_{\widehat{F}_p}\). The \emph{framed RS hierarchy} is the central subalgebroid of the algebroid of RS tweakings generated by the restrictions of \(x_p^i\) for all \(p\), and \(i \in \mathbb Z_{\geq 0}\)
\end{defn}

In order to describe the action of the framed RS hierarchy, we need to understand the tweaking action on spectral sheaves as a Hitchin type flow. The hierarchy will descend to deformations of the Higgs fields associated to spectral sheaves, and we now turn to describing the induced action. 

The last ingredient before we can turn to the main theorem of this section is the map \(\mathbf{H}\) which will exhibit \(\mathsf{RS}_{\sigma, n}(E, V)\) as an algebraically completely integrable Hamiltonian system.

Let 
\[
\Hitchin_{RS}(E) = \bigoplus_{i=0}^\infty H^0(E, \OO_E(D_0 + D_\infty)^{\otimes(i+1)})
\]
be the Hitchin base for the RS phase space. Define a map
\[
\mathbf H: \mathsf{RS}_{\sigma, n}(E, V) \lra \Hitchin_{RS}(E)
\]
which in components is given by
\[
H_i(W, \eta_0, \eta_\infty) = \frac{1}{i+1} \Tr\left( (\eta_0 \cdot \eta_\infty)^i\right).
\]

\begin{prop}\label{prop: tweaking theorem}
	The vector fields of the framed RS hierarchy of Definition \ref{def: RS hierarchy} agree with the Hamiltonian vector fields defined by the map \(\mathbf H\). 
\end{prop}	
\begin{proof}
	We will prove the theorem by calculating the flows given by homogeneous elements of \(\Hitchin_{RS}(E)\). Fix a \(\xi \in H^0(E, \OO_E(D_0 + D_\infty)^i)^* \cong H^1(E, \OO_E(-D_0-D_\infty)^i)\) where the duality is given by the usual residue pairing. Therefore the function on \(\mathsf{RS}_{\sigma, n}(E,V)\) defined by \(\xi\) is given by
	\[
	H_\xi(W, \eta_0, \eta_\infty) = \Res\left(\xi \Tr \left( \frac{1}{i+1} (\eta_0 \cdot \eta_\infty)^{i+1}\right)\right).
	\]

	Furthermore let \((s, \theta) \in H^1(E, \End(W)(-D_\infty - D_0) \oplus H^0(E, \End(W) \otimes \mathcal L_\sigma(D_0))\) represent a tangent vector \(V_{(s, \theta)}\) in \(T_{(W, \eta_0, \eta_\infty)}\). A simple calculation over \(\mathbb C[\varepsilon]/(\varepsilon^2)\) yields
	\begin{align*}
		dH_\xi|_{(W, \eta_0, \eta_\infty)}(s, \theta) &= V_{(s, \theta)}H_\xi(W, \eta_0, \eta_\infty) = \left. \frac{d}{d \varepsilon}\right|_{\varepsilon = 0}H_\xi(W, \eta_0 + \varepsilon \theta, \eta_\infty) = \\
		& = \left. \frac{d}{d \varepsilon}\right|_{\varepsilon = 0} \Res( \xi \Tr(\frac{1}{i+1}((\eta_0 + \varepsilon \theta) \cdot \eta_\infty)^{i+1})) \\
		& = \left. \frac{d}{d \varepsilon} \right|_{\varepsilon = 0} \Res( \xi \Tr((\eta_0 \cdot \eta_\infty)^{i+1} + \varepsilon (i+1)(\eta_0 \cdot \eta_\infty)^i(\theta \eta_\infty))) \\
		& = \Res( \xi \Tr ((\eta_0 \cdot \eta_\infty)^i (\eta_\infty \theta)))
	\end{align*}
	Now we calculate the Hamiltonian vector field corresponding to \(\xi\) at \((W, \eta_0, \eta_\infty)\), denoted \(v_\xi|_{(W, \eta_0, \eta_\infty)} = (s_\xi, \theta_\xi)\). The Hamiltonian vector field is a solution to the equation 
	\[
	-\omega|_{(W, \eta_0, \eta_\infty)}((s, \theta), (s_\xi, \theta_\xi)) = dH_\xi|_{(W, \eta_0, \eta_\infty)}(s, \theta)
	\]
	for every \((s, \theta)\). Equating the two sides
	\[
	- \Res \Tr ( s \cdot \eta_0 \cdot \theta_\xi - s_\xi \cdot \eta_\infty \cdot \theta) = \Res (\xi \Tr ((\eta_0 \cdot \eta_\infty)^i(\eta_\infty \theta)))
	\]
	we find that the unique solution is \((s_\xi, \theta_\xi) = (\xi (\eta_0 \cdot \eta_\infty)^i, 0)\). 

	On the tweaking flow side of the story, let \(\tilde \xi\) represent a homogeneous degree \(i\) element of \(\mathbb C[x_p] \subset \OO_{\widehat{F}_p}\) which descends to the deformation defined by \(\xi\) in \(\Hitchin_{RS}(E)\). We realize the RS spectral sheaf \(\mathscr F\) restricted to the \(\mathbb C^*\) bundle \(S_\sigma^*\) as a \(\OO_{S_\sigma^*}\)-module via the endomorphism \(\eta_0 \cdot \eta_\infty\). Hence the tweaking flow modifies the \(\OO_{S_\sigma^*}\)-module structure by multiplication by \(\tilde \xi\). 

	Therefore, after pushing forward the deformation defined by \(\tilde \xi\) to a deformation of \((W, \eta_0, \eta_\infty)\), the vector field on \(\mathsf{RS}_{\sigma, n}(E, V)\) defined by \(\tilde \xi\) yields the same element \(\xi (\eta_0 \cdot \eta_\infty)^i \in \End(W)(-D_0 - D_\infty)\). 
\end{proof}

\begin{cor}
	For spin framing \(V = \OO_{q_0}^k \oplus \OO_{q_\infty}^k\), the flows of the spin RS hierarchy agree with the flows defined by the spin RS Hamiltonians. 
\end{cor}

\begin{proof}
	By Proposition \ref{prop:geometric lax matrix}, the composition \(\eta_0 \circ \eta_\infty\) is the Lax matrix for the RS system. Hence the function \(H_i = \Tr (\eta_0 \cdot \eta_\infty)^i\) is the \(i\)th Hamiltonian for the RS system, and the Hamiltonian vector fields defined by the RS hierarchy match up with the vector fields coming from the RS Hamiltonians. 
\end{proof}

In particular, the above corollary yields the desired spectral description of the RS phase space. 

\begin{thm}\label{thm:main theorem}
	The moduli space \(\mathsf{RS}_{\sigma, n}(E, \OO_{q_0}^k \oplus \OO_{q_\infty}^k)\) is isomorphic to a completion of the spin Ruijsenaars-Schneider phase space, with RS flows given by tweaking flows on RS spectral sheaves. 
\end{thm}

\section{Comparison with the Calogero-Moser system} \label{sec:comparison}
Our description of the RS phase space in terms of spectral curves living in \(S_\sigma = \mathbb P(\OO_E \oplus \mathcal L_\sigma)\) parallels the spectral description of the CM phase space in \cite{MR2377220}. The authors define the CM phase space \(\mathsf{CM}_n(E, V)\) as the moduli space of framed spectral sheaves living in the ruled surface \(\overline{E}^\natural\) living over a Weierstrass cubic curve \(E\). The surface is the projectivization of the Atiyah bundle \(\mathcal A\) obtained as the unique non-trivial extension of \(\OO_E\) by itself 
\[
0 \lra \OO_E \lra \mathcal A \lra \OO_E \lra 0.
\]
The extension being non-trivial means that \(\overline{E}^\natural\) has only one section at infinity, so the framing data \(V\) in the definition of the CM phase space is a torsion sheaf supported on the smooth part of \(E_\infty\). The precise definition of \(\mathsf{CM}_n(E, V)\) is then very familiar as the moduli space of pure one dimensional sheaves in \(\overline{E}^\natural\) framed by \(V\) at \(E_\infty\) together with finiteness conditions guaranteeing the pushfoward to \(E\) is a semi-stable vector bundle of appropriate degree. The CM hierarchy on the phase space is also described in terms of tweaking flows of CM spectral sheaves along the intersection divisor at infinity. 

We now turn to using the similarity of the CM and RS descriptions to intrepret well-known relations between the hierarchies.

\subsection{The universal RS system}\label{subsec:Universal RS system}
The description of the RS system as a relativistic analogue of the CM system can be made precise in the \(c \rightarrow \infty\) limit of Hamiltonians as described in the introduction. More generally, it is known that the whole RS hierarchy can degenerate to the CM hierarchy in the same limit by taking the limit of the RS Lax matrix. Our goal in this section is to interpret this degeneration in our geometric perspective.

The first step is to find an appropriate setting in which the limit of the bundles \(\OO_E \oplus \mathcal L_\sigma \lra \mathcal A\) as \(\sigma\) approaches the basepoint makes sense. We will do this by constructing the universal rank 2 bundle \(\mathcal R\) over \(E \times \overline{\Jac} E \cong E \times E\) which equals \(\OO_E \oplus \mathcal L_\sigma\) when restricted to \(E \times \{\sigma\}\), and over \(E \times \{b\}\) equals \(\At\). We offer two constructions of this bundle.
\begin{construction}
	The space \(\overline{E}^\natural\) is birational to the trivial \(\mathbb P^1\) bundle on \(E\). We now recall the explicit birational transformations taking one to the other. 

	Begin with the trivial \(\mathbb P^1_E = \mathbb P(\OO_E \oplus \OO_E)\) bundle. The divisor class group of \(\mathbb P^1_E\) is generated by the classes of fibers \(F_q\) for \(q \in E\) and the section at infinity \(E_\infty\), with intersection product given by
	\[
		F_q. F_{q'} = 0, \,\,\, E_\infty.E_\infty = 0,\,\,\, E_\infty.F_q = 1 
	\]
	for all \(q, q'\). Let \(b_0\) be the point of intersection of the fiber over the basepoint \(F_b\) and \(E_0\). Let \(\widetilde X\) be the blowup of \(\mathbb P^1_E\) at the point \(b_0\). 
	We have that \(\Cl \widetilde{X} \cong \Cl \mathbb P^1_E \oplus \mathbb Z Z_{b_0}\) where \(Z_{b_0}\) is the exceptional divisor over \(b_0\). The effective divisor \(E_0 + Z_{b_0}\) has self intersection \(-1\) so it can be blown down to yield a surface \(X\) which we claim is isomorphic to \(\overline{E}^\natural\). 

	An indirect way of seeing the claim is that \(X\) is a \(\PP^1\) bundle on \(E\) with the section given by transporting \(E_\infty\) to \(X\). The complement of the section \(E_\infty \subset X\) is the total space of a rank 2 vector bundle with no non-zero sections, hence we can identify it with the total space of \(\mathcal A\) and \(X \cong \overline{E}^\natural\). 

	We now repeat this construction in families over \(E \times \overline \Jac E\). Let \(\mathcal P\) be the Poincar\'e bundle over the product space, and let \(\underline \OO_E\) be the structure sheaf. The notation for the structure sheaf is to emphasize that it is the constant line bundle on \(E \times \overline \Jac E\) which restricts to \(\OO_E\) on each fiber \(E \times \{\sigma\}\) for every \(\sigma\). Now consider the \(\PP^1\) bundle 
	\[
	\PP(\underline \OO_E \oplus \mathcal P) \lra E \times \overline \Jac E.
	\]
	In order to obtain \(\PP(\mathcal R)\) now simply repeat the blow up, blow down construction on the fiber over \(E \times \{b\}\). The transformation does not change the space over any other fiber, which yields the desired \(\PP(\OO_E \oplus \mathcal L_\sigma)\). But over \(E \times \{b\}\) we get the space \(\overline{E}^\natural\). 

\end{construction}
\begin{construction}
	Here we provide an alternative construction of \(\mathcal R\). Let \(\mathcal I_\Delta\) be the ideal sheaf of the diagonal in \(E \times \overline{\Jac} E \cong E \times E\), and let \(\mathcal I_b\) be the ideal sheaf of divisor \(\{b\} \times \overline \Jac E \subset E \times \overline \Jac E\). The ideal sheaf of the union \(\mathcal I_D = \mathcal I_\Delta \cdot \mathcal I_b\) fits into a short exact sequence 
	\[
	0 \lra \mathcal I_D \lra \OO_{E \times \overline \Jac E} \lra Q \lra 0.
	\]
	The fiber of \(Q\) over any point \(\sigma \in \overline \Jac E\), \(\sigma \neq b\), equals the direct sum of skyscraper sheaves over the points \(\sigma\) and \(b\) in \(E\). On the other hand, over \(b \in \overline \Jac E\) the sheaf \(Q\) restricts to non-split length 2 torsion sheaf supported at the point \(b \in E\). Now consider the diagram
	\begin{center}
	\begin{tikzcd}[column sep = tiny]
		E \times \overline \Jac E & & E \times \overline \Jac E \times \overline \Jac E \arrow[ll, "\pi_{12}"] \arrow[dr, "\pi_{23}"] \arrow[dl, "\pi_{13}"] & \\
		& E \times \overline \Jac E & & \overline \Jac E \times \overline \Jac E \cong E \times \overline \Jac E
	\end{tikzcd}
	\end{center}
	Let \(\mathcal P_{12} = \pi_{12}^* \mathcal P\) be the relative Poincar\'e bundle on the triple product. The bundle \(\mathcal P_{12}\) then defines a relative Fourier-Mukai transform
	\[
	\mathcal{FM}(\mathscr G) = (\pi_{23})_*(\pi_{13}^* \mathscr G \otimes \mathcal P_{12})
	\]
	which when restricted to a fiber of the second product \(\overline \Jac E\) in the triple product agrees with the usual Fourier-Mukai transform on elliptic curves. Consider now the sheaf \(\mathcal{FM}(Q)\). Analyzing the bundle fiber by fiber again, when restricted to \(E \times \{\sigma\}\) for \(\sigma \neq b\) the sheaf is the Fourier-Mukai transform of the direct sum skyscraper sheaves supported at \(b\) and \(\sigma\) which equals \(\OO_E \oplus \mathcal L_\sigma\). Over \(b \in \overline \Jac E\) though, we have the Fourier-Mukai transform of the non-split length 2 torsion sheaf supported at \(b\). The resulting sheaf is then a rank 2 bundle on \(E\) which is a non-split extension of \(\OO_E\) by itself which is precisely the Atiyah bundle. Therefore \(\mathcal{FM}(Q)\) recovers the desired universal rank 2 bundle \(\mathcal R\). 
\end{construction}
In both the above constructions, it is immediately clear that \(\PP(\mathcal R)\) has a well-defined section at infinity \((E \times \overline \Jac E)_\infty \). We will now define a space of sheaves in the \(\PP^1\) bundle \(\PP(\mathcal R)\) framed along \((E \times \overline \Jac E)_\infty\) that will serve as the phase space for the universal RS system
\begin{defn}\label{def:universal RS}
	Let \(\mathcal V\) be a torsion sheaf on \(\mathbb P(\mathcal R)\) such that is supported on the subscheme \(D_{\mathcal V}\) contained in the divisor \((E \times \overline \Jac E)_\infty\) such that \(D_{\mathcal V}\) has finite support over the projection to \(E\). Let \(\mathsf{RS}_n(E, \mathcal V)\) be the moduli space of sheaves \(\mathscr F\) in \(\mathbb P(\mathcal R)\) satisfying
	\begin{itemize}
		\item \(\mathscr F\) is a torsion sheaf of pure dimension 2 with support \(\widetilde{\Sigma}\) such that the restriction of the projection to \(E\) is finite.
		\item If \(\pi\) denotes the projection of \(\PP(\mathcal R)\) to the base, then \(\pi_* \mathscr F\) is a rank \(n\) vector bundle on \(E \times \overline \Jac E\) such that along the fibers of the projection to \(\overline \Jac E\) the vector bundle is semi-stable, and
		\[
		\deg \pi_*(\mathscr F ( k (E \times \overline \Jac E)_\infty)) = (k + 1) \deg \mathcal V.
		\]
		\item The spectral sheaf \(\mathscr F\) is framed along the section at infinity by the torsion sheaf \(\mathcal V\), and over the open complement of \(E \times \{b\}\) where the \(\PP(\mathcal R)\) has a zero-section, the sheaf \(\mathscr F\) is framed by the torsion sheaf \(\alpha_n^*\mathcal V\) supported on the zero-section. 
	\end{itemize}
\end{defn}

Note that fiberwise, this definition for \(\sigma \neq b \in \overline \Jac E\) recovers the original definition of \(\mathsf{RS}_{\sigma, n}(E, V)\) where \(V\) satisfies the additional assumptions of Remark \ref{rmk:framing}, and over \(b\) the sheaves of the moduli space restrict to the CM spectral sheaves of \cite{MR2377220}. The general constructions of this paper also imply that the Hitchin description of this moduli space would be given by a prolonged Higgs field on the vector bundle \(\mathcal W = \pi_* \mathscr F\) on the space \(E \times \overline \Jac E\)
\[
\eta: \mathcal W \lra \mathcal W \otimes \mathcal P(D_{\mathcal V}).
\]
Away from the fiber over \(b \in \overline \Jac E\), there would also exist a second Higgs field coming from the open set containing the \(0\)-section, but the compatibility relation would ensure that it is determined up to a shift by the inverse of \(\eta\). We can define the map to the universal Hitchin base
\[
\mathbf H_U: \mathsf{RS}_n(E, \mathcal V) \lra \bigoplus_{i} H^0(E \times \overline \Jac E, \OO_{E \times \overline \Jac E}(D_{\mathcal V}+D_{\alpha_n^* \mathcal V})^{\otimes i})
\]
whose \(i\)th component is given by 
\[
H_i(W, \eta) = \Tr \frac{1}{1+i} (\eta(z) \cdot \eta^{-1}(z + y))^{i+1}
\]
with \(z\) and \(y\) the coordinates along \(E\) and \(\overline \Jac E\) respectively. 

The general setting in which tweaking of spectral sheaves was introduced allows us to define the universal RS hierarchy of flows to be the flows obtained by restrictions of germs of Laurent series along the support divisor of \(\mathcal V\) to the spectral surface \(\widetilde \Sigma\), and tweaking with respect to those polynomials. The fiberwise constructions of this paper and the previous results on the CM system readily imply that 


\begin{prop}\label{prop: universal RS}
	The moduli space \(\mathsf{RS}_n(E, \mathcal V)\) together with the map \(\mathbf{H}_U\) defines a completely integrable system relative to \(\overline \Jac E\), with Hamiltonian flows being given by tweaking of spectral sheaves along the divisor \(D_{\mathcal V}\). 
\end{prop}
\begin{proof}
	Fiberwise the result has been proven by a combination of Theorem \ref{thm:main theorem} and the main theorem of \cite{MR2377220}, but what remains to be checked is the compatibility of the flows between \(\mathsf{RS}_{\sigma, n}(E, V_0 \oplus V_\infty)\) and \(\mathsf{CM}_n(E, V)\) as \(\sigma \lra b\). The term \(\hbar = \sigma - b\) appearing in the formula for the Lax matrix of the RS system plays the role of the constant \(c^{-1}\) in the formula for the factorized RS Lax matrix. In \cite{Vasilyev:2018cyt} the degeneration of the factorized RS Lax matrix is shown to coincide with the factorized CM Lax matrix
	\[
	L^{CM} = \diag(p_1, \ldots, p_n) + n \Xi^{-1}(z)\Xi'(z).
	\]

	With a modification of the results in \cite{MR2377220} with the Hamiltonian function being given instead by the factorized Lax matrix rather than the classical CM Lax matrix, we see that the map \(\mathbf H_U\) defines a universal Hamiltonian map for both the RS (\(\sigma \neq b\)) and CM (\(\sigma = b\)) hierarchies to the Hitchin base realizing the tweaking flows as Lax flows on the respective integrable systems.
\end{proof}

\begin{rmk}
	The integrable system \(\mathsf{RS}_n(E, \mathcal V)\) is the ``correct'' setting to study the RS system, as it treats the integrable system independently of the choice of \(\sigma\), and on an equal footing with the CM hierarchy. 
\end{rmk}

\subsection{Geometric description of Ruijsenaars duality}\label{subsec: TCM-RRS}
The other relation between the RS and CM systems that our geometric perspective can shed some light on is the trigonometric CM-rational RS duality. Ruijsenaars duality is a duality between integrable systems which is given by a symplectomorphism \(\Psi\) between their phase spaces \(M_1, M_2\) which swaps the position and action coordinates on the two spaces
\[
(\mathbf{q}_1, \mathbf I_1) \overset{\Psi}{\leftrightarrow} (\mathbf I_2,\mathbf{q}_2).
\]
It has been known since the work of Kazhdan, Kostant and Sternberg that the rational CM system is self-dual \cite{doi:10.1002/cpa.3160310405}. Their description of the rational CM phase space as a Hamiltonian reduction on \(T^* \mathfrak{gl}_n\) immediately yields the self-duality as two different choices of gauge slice for the moment map. They also describe the trigonometric CM system as a Hamiltonian reduction on the space \(T^* GL_n\), and the dual theory obtained in the same way was later recognized as the rational RS system. 

Ruijsenaars studied the duality between the many particle systems of CM type and their relativistic generalizations of RS type in \cite{ruijsenaars1988} by calculating the action-angle variables and explicitly giving the symplectomorphism swapping the coordinates. Since then the duality has been studied in the context of Seiberg-Witten theory as a duality between 4D and 5D \(\mathcal N = 2\) SUSY gauge theory \cite{Braden:1999zpa}, \cite{Braden:1999aj}.

In the trigonometric and rational limits, the integrable systems are described in terms of the base Weierstrass cubic which limiting to the nodal and cuspidal cubic respectively. We can first relate the two ambient spaces in which the spectral curves live by pulling back to the normalization. First, consider the nodal cubic \(E_{node}\), and its associated space \(\overline{E}_{node}^\natural\). When pulling back to the normalization \(\PP^1 \rightarrow E_{node}\) we obtain
\begin{center}
	\begin{tikzcd}
		\overline{E}_{node}^\natural \arrow[d] & \arrow[l] \PP^1 \times \PP^1 \arrow[d,"p_1"] \\
		E_{node} & \arrow[l] \PP^1
	\end{tikzcd}
\end{center}
Choose the normalization map so that \(0, \infty \in \PP^1\) are the points that maps to the node, then the fibers \(F_\infty , F_0 \subset \PP^1 \times \PP^1\) are distinguished divisor of type \((1,0)\). There is also \(\PP^1_\infty = (\PP^1 \times \{\infty\})\) which is a distinguished divisor of type \((0,1)\). Applying the transposition \(\tau: \PP^1 \times \PP^1 \rightarrow \PP^1 \times \PP^1\) swapping the two factors of \(\PP^1\), the divisors also swap between types \((1,0)\) and \((0,1)\). We can then descend to the bundle \(S_\sigma\) over the cuspidal curve by mapping \(\tau(\PP^1_\infty)\) to the cusp point, and the images \(\tau(F_0), \tau(F_\infty)\) mapping to the sections \(E_0, E_\infty\). This is all packaged in the diagram 

\begin{center}
	\begin{tikzcd}
		\overline{E}_{node}^\natural \arrow[d] & \arrow[l] \PP^1 \times \PP^1 \arrow[d,"p_1"]  \arrow[r, "\tau"] & \PP^1 \times \PP^1 \arrow[r] \arrow[d, "p_1"] & S_{\sigma, cusp} \arrow[d] \\
		E_{node} & \arrow[l] \PP^1 & \PP^1 \arrow[r] & E_{cusp} \\
		F_\infty & \arrow[l, mapsto] F_\infty \cup F_0 \arrow[r, "\tau"] & \tau(F_0) \cup \tau(F_\infty) \arrow[r, mapsto] & E_\infty \cup E_0 \\
		E_\infty & \arrow[l, mapsto] \PP^1_\infty \arrow[r, "\tau"] & \tau(\PP^1_\infty) \arrow[r, mapsto] & F_\infty. \\
	\end{tikzcd}
\end{center}

In order to account for spectral curves living in the \(\PP^1\) bundles we need to introduce an additional twist. For example, consider an RS spectral curve \(\widetilde{\Sigma}\) which meets the sections \(E_0\) and \(E_\infty\) at the points \(q_0\) and \(q_\infty\). We know that these points are translates by \(n(\sigma -b)\). Therefore, in order for \(\tau^{-1}\) of the pullback of \(\widetilde{\Sigma}\) to the normalization to descend to a curve in \(\overline{E}^\natural_{node}\) we need to shift the second \(\PP^1\) by the unique map \(T_\sigma\) which fixes the intersection points of \(\PP^1_\infty\) and \(F_\infty \cup F_0\), but maps \(q_0\) to \(q_\infty\). Therefore we modify the above diagram by instead using the map
\[
\tilde{\tau}(x, y) = (y, T_\sigma x).
\] 

The naive conjecture about the action of Ruijsenaars duality on spectral curves is that they transform into each other via this transformation. Unfortunately, this is not the correct picture. In fact, take for example the spinless CM spectral curves for \(n\) particles which live in the linear series \(|n E_\infty + F|\). Under the above transformation, this linear series gets sent to the linear series \(| E_\infty + n F|\) which would correspond to one RS particle with spin living in an \(n\)-dimensional vector space. 

A fix to the naive construction is to consider the CM spectral curve as the unique curve in its linear series passing through the points \((q_1, p_1)\), ... , \((q_n, p_n)\) in \(\overline{E}^\natural_{node}\) where \(q_i,  p_i\) are the initial positions and momenta of the CM particles. Under the above transformation, these points will get sent to points \((q_1',\theta_1)\), ..., \((q_n',\theta_n)\) in \(S_\sigma\), and there will be a unique RS spectral curve in the appropriate linear series passing throug these points. The sheaves can also be reconstructed from the above points, at least generically for spinless framing, as the unique rank 1 torsion free sheaves supported on the spectral curves with generic section vanishing at the fixed points. 

The above construction therefore produces two dualities between the trigonometric CM and rational RS systems. One of which is definitely not the Ruijsenaars duality, but still appears to be of independent interest as it relates systems with different number of particles and framing. The second duality described above has a chance of being the right one, but explicit calculations of the action-angle variables for both the CM and RS spectral sheaves would be required to show that the duality transforms the coordinates appropriately. 

Understanding the action of Ruijsenaars duality on spectral curves in this geometric picture could lay the groundwork to understanding the integrable systems dual to the elliptic CM and elliptic RS. Whatever these integrable systems are, they should be degenerations of a conjectural doubly-elliptic integrable system which serves as a master particle system of which all the CM and RS systems appear as degenerations.

\bibliographystyle{alpha}

\bibliography{References}

\begin{thebibliography}{BMMM00}

\bibitem[AMM77]{doi:10.1002/cpa.3160300106}
H.~Airault, H.~P. McKean, and J.~Moser.
\newblock Rational and elliptic solutions of the korteweg-de vries equation and
  a related many-body problem.
\newblock {\em Communications on Pure and Applied Mathematics}, 30(1):95--148,
  1977.

\bibitem[BL94]{beauville1994}
Arnaud Beauville and Yves Laszlo.
\newblock Conformal blocks and generalized theta functions.
\newblock {\em Comm. Math. Phys.}, 164(2):385--419, 1994.

\bibitem[BMMM99]{Braden:1999zpa}
H.~W. Braden, A.~Marshakov, A.~Mironov, and A.~Morozov.
\newblock {The Ruijsenaars-Schneider model in the context of Seiberg-Witten
  theory}.
\newblock {\em Nucl. Phys.}, B558:371--390, 1999.

\bibitem[BMMM00]{Braden:1999aj}
H.~W. Braden, A.~Marshakov, A.~Mironov, and A.~Morozov.
\newblock {On double elliptic integrable systems. 1. A Duality argument for the
  case of SU(2)}.
\newblock {\em Nucl. Phys.}, B573:553--572, 2000.

\bibitem[Bot95]{MR1346215}
Francesco Bottacin.
\newblock Poisson structures on moduli spaces of sheaves over {P}oisson
  surfaces.
\newblock {\em Invent. Math.}, 121(2):421--436, 1995.

\bibitem[Bot98]{MR1660136}
Francesco Bottacin.
\newblock Poisson structures on {H}ilbert schemes of points of a surface and
  integrable systems.
\newblock {\em Manuscripta Math.}, 97(4):517--527, 1998.

\bibitem[BZF01]{PMIHES_2001__94__87_0}
David Ben-Zvi and Edward Frenkel.
\newblock Spectral curves, opers and integrable systems.
\newblock {\em Publications Math\'ematiques de l'IH\'ES}, 94:87--159, 2001.

\bibitem[BZN07]{MR2377220}
David Ben-Zvi and Thomas Nevins.
\newblock Flows of {C}alogero-{M}oser systems.
\newblock {\em Int. Math. Res. Not. IMRN}, (23):Art. ID rnm105, 38, 2007.

\bibitem[BZN11]{MR2835323}
David Ben-Zvi and Thomas Nevins.
\newblock {$\mathscr D$}-bundles and integrable hierarchies.
\newblock {\em J. Eur. Math. Soc. (JEMS)}, 13(6):1505--1567, 2011.

\bibitem[BZNP]{MPTNDBZ}
David Ben-Zvi, Thomas Nevins, and Matej Penciak.
\newblock Toda lattice hierarchy and noncommutative geometry.
\newblock {\em in preparation}.

\bibitem[{Cal}71]{1971JMP....12..419C}
F.~{Calogero}.
\newblock Solution of the one-dimensional n-body problems with quadratic and/or
  inversely quadratic pair potentials.
\newblock {\em Journal of Mathematical Physics}, 12:419--436, March 1971.

\bibitem[DM96]{Donagi1996}
Ron Donagi and Eyal Markman.
\newblock {\em Spectral covers, algebraically completely integrable,
  hamiltonian systems, and moduli of bundles}, pages 1--119.
\newblock Springer Berlin Heidelberg, Berlin, Heidelberg, 1996.

\bibitem[FWM01]{descentdata}
R~Friedman and J~W.~Morgan.
\newblock Holomorphic principal bundles over elliptic curves iii: Singular
  curves and fibrations.
\newblock 09 2001.

\bibitem[Has97]{Hasegawa1997}
Koji Hasegawa.
\newblock Ruijsenaars' commuting difference operators as commuting transfer
  matrices.
\newblock {\em Comm. Math. Phys.}, 187(2):289--325, Aug 1997.

\bibitem[JMO87]{MR908997}
Michio Jimbo, Tetsuji Miwa, and Masato Okado.
\newblock Solvable lattice models whose states are dominant integral weights of
  {$A^{(1)}_{n-1}$}.
\newblock {\em Lett. Math. Phys.}, 14(2):123--131, 1987.

\bibitem[KKS78]{doi:10.1002/cpa.3160310405}
D.~Kazhdan, B.~Kostant, and S.~Sternberg.
\newblock Hamiltonian group actions and dynamical systems of calogero type.
\newblock {\em Communications on Pure and Applied Mathematics}, 31(4):481--507,
  1978.

\bibitem[KOP]{duality}
Ludmil Katzarkov, Dmitri Orlov, and Tony Pantev.
\newblock Notes on {H}iggs bundles and {D}-branes.

\bibitem[Kri80]{Krichever1980}
I.~M. Krichever.
\newblock Elliptic solutions of the {K}adomtsev-{P}etviashvili equation and
  integrable systems of particles.
\newblock {\em Functional Analysis and Its Applications}, 14(4):282--290, Oct
  1980.

\bibitem[KZ95]{MR1379076}
Igor~Moiseevich Krichever and A.~Zabrodin.
\newblock Spin generalization of the {R}uijsenaars-{S}chneider model, the
  nonabelian two-dimensionalized {T}oda lattice, and representations of the
  {S}klyanin algebra.
\newblock {\em Uspekhi Mat. Nauk}, 50(6(306)):3--56, 1995.

\bibitem[Lau87]{PMIHES_1987__65__131_0}
G{\'e}rard Laumon.
\newblock Fourier transform, constants of functional equations and weil
  conjecture.
\newblock {\em Mathematical Publications of the IH{\'E}S}, 65:131--210, 1987.

\bibitem[Law89]{MR1007595}
Derek~F. Lawden.
\newblock {\em Elliptic functions and applications}, volume~80 of {\em Applied
  Mathematical Sciences}.
\newblock Springer-Verlag, New York, 1989.

\bibitem[Lax]{doi:10.1002/cpa.3160210503}
Peter~D. Lax.
\newblock Integrals of nonlinear equations of evolution and solitary waves.
\newblock {\em Communications on Pure and Applied Mathematics}, 21(5):467--490.

\bibitem[Mos75]{Moser}
J~Moser.
\newblock Three integrable hamiltonian systems connected with isospectral
  deformations.
\newblock {\em Advances in Mathematics}, 16:197--220, 05 1975.

\bibitem[Mum07]{MR2352717}
David Mumford.
\newblock {\em Tata lectures on theta. {I}}.
\newblock Modern Birkh\"{a}user Classics. Birkh\"{a}user Boston, Inc., Boston,
  MA, 2007.
\newblock With the collaboration of C. Musili, M. Nori, E. Previato and M.
  Stillman, Reprint of the 1983 edition.

\bibitem[Nek96]{nekrasov1996}
Nikita Nekrasov.
\newblock Holomorphic bundles and many-body systems.
\newblock {\em Comm. Math. Phys.}, 180(3):587--603, 1996.

\bibitem[Nit91]{doi:10.1112/plms/s3-62.2.275}
Nitin Nitsure.
\newblock Moduli space of semistable pairs on a curve.
\newblock {\em Proceedings of the London Mathematical Society},
  s3-62(2):275--300, 1991.

\bibitem[Obl]{Oblomkov_doubleaffine}
Alexei Oblomkov.
\newblock Double affine {H}ecke algebras and {C}alogero-{M}oser spaces.
\newblock {\em Represent. Theory}, pages 243--266.

\bibitem[Pam07]{Pampolina}
John Pampolina.
\newblock Integrable systems and the spin {R}uijsenaars-{S}chneider model.
\newblock 08 2007.

\bibitem[PR01]{polishchuk2001}
A.~Polishchuk and M.~Rothstein.
\newblock Fourier transform for {D}-algebras, i.
\newblock {\em Duke Math. J.}, 109(1):123--146, 07 2001.

\bibitem[RS86]{MR851627}
S.~N.~M. Ruijsenaars and H.~Schneider.
\newblock A new class of integrable systems and its relation to solitons.
\newblock {\em Ann. Physics}, 170(2):370--405, 1986.

\bibitem[Rui87]{ruijsenaars1987}
S.~N.~M. Ruijsenaars.
\newblock Complete integrability of relativistic {C}alogero-{M}oser systems and
  elliptic function identities.
\newblock {\em Comm. Math. Phys.}, 110(2):191--213, 1987.

\bibitem[Rui88]{ruijsenaars1988}
S.~N.~M. Ruijsenaars.
\newblock Action-angle maps and scattering theory for some finite-dimensional
  integrable systems. i. the pure soliton case.
\newblock {\em Comm. Math. Phys.}, 115(1):127--165, 1988.

\bibitem[TV90]{Treibich1990}
A.~Treibich and J.-L. Verdier.
\newblock {\em Solitons elliptiques}, pages 437--480.
\newblock Birkh{\"a}user Boston, Boston, MA, 1990.

\bibitem[TV93]{TV2}
A.~Treibich and J.-L. Verdier.
\newblock {\em Vari\'et\'es de Kritchever des Solitons Elliptiques de KP},
  pages 187--232.
\newblock Hindustan Book Agency, 1993.

\bibitem[UT84]{ueno1984}
Kimio Ueno and Kanehisa Takasaki.
\newblock Toda lattice hierarchy.
\newblock In {\em Group Representations and Systems of Differential Equations},
  pages 1--95, Tokyo, Japan, 1984. Mathematical Society of Japan.

\bibitem[Vak99]{math/9909079}
V.~Vakulenko.
\newblock Note on the {R}uijsenaars-{S}chneider model, 1999.

\bibitem[VZ18]{Vasilyev:2018cyt}
M.~Vasilyev and A.~Zotov.
\newblock {On factorized Lax pairs for classical many-body integrable systems}.
\newblock 2018.

\bibitem[Yam08]{10.1093/imrp/rpn008}
Daisuke Yamakawa.
\newblock Geometry of multiplicative preprojective algebra.
\newblock {\em International Mathematics Research Papers}, 2008, 01 2008.

\end{thebibliography}

\end{document}